\RequirePackage{fix-cm}
\documentclass[smallextended]{svjour3}       
\smartqed  
\usepackage{graphicx}
\usepackage{graphicx}
\usepackage{graphicx,amssymb,amsmath}
\usepackage{epsfig}
\usepackage{subfigure}
\usepackage{cite}
\usepackage{booktabs}
\usepackage{array}
\usepackage{color}
\usepackage{enumerate}
\usepackage{algorithm}
\usepackage{algorithmic}
\usepackage{float}
\usepackage{multirow}
\usepackage{rotating}
\usepackage[misc]{ifsym}
\usepackage{bm}
\usepackage{mathrsfs}
\usepackage[pdfstartview=FitH,colorlinks,pdfborder=001,linkcolor=blue,anchorcolor=blue,citecolor=blue]{hyperref}
\usepackage[top=1.3in, bottom=1.0in, left=1.30in, right=1.30in]{geometry}
\usepackage{mathrsfs}
\usepackage[misc]{ifsym} 
\newtheorem{assumption}{Assumption}

\begin{document}

\title{An FE-dABCD algorithm for elliptic optimal control problems with constraints on the gradient of the state and control.\thanks{This study was funded by the National Natural Science Foundation of China (Nos. 11571061, 11401075, 11701065, 11501079), China Postdoctoral Science Foundation (No. 2018M632018) and the Fundamental Research Funds for the Central Universities (No. DUT16LK36).}
}

\titlerunning{An FE-dABCD algorithm for optimal control problems with constraints on state and control}

\author{Zixuan Chen$^{1}$         \and
        Xiaoliang Song$^{1}$\and
        Bo Yu$^{1,2}$\and
        Xiaotong Chen$^{1}$ 
}

\authorrunning{Zixuan Chen et al.}

\institute{Zixuan Chen\at
               \email{chenzixuan@mail.dlut.edu.cn}
               \and
              Xiaoliang Song\ \Letter \at
              \email{songxiaoliang@mail.dlut.edu.cn}
              \and
              Bo Yu\at
              \email{yubo@dlut.edu.cn}
              \and
              Xiaotong Chen\at
              \email{chenxiaotong@mail.dlut.edu.cn}
              \and
              $1.$\ School of Mathematical Sciences, Dalian University of Technology, Dalian 116025, Liaoning, China\\
              $2.$\ School of Mathematical and Physics Science, Dalian University of Technology, Panjin 124200, Liaoning, China
}

\date{Received: date / Accepted: date}

\maketitle

\begin{abstract}
In this paper, elliptic control problems with integral constraint on the gradient of the state and box constraints on the control are considered. The optimal conditions of the problem are proved. To numerically solve the problem, we use the {\emph{First discretize, then optimize}} approach. Specifically, we discretize both the state and the control by piecewise linear functions. To solve the discretized problem efficiently, we first transform it into a multi-block unconstrained convex optimization problem via its dual, then we extend the inexact majorized accelerating block coordinate descent (imABCD) algorithm to solve it. The entire algorithm framework is called finite element duality-based inexact majorized accelerating block coordinate descent (FE-dABCD) algorithm. Thanks to the inexactness of the FE-dABCD algorithm, each subproblems are allowed to be solved inexactly. For the smooth subproblem, we use the generalized minimal residual (GMRES) method with preconditioner to slove it. For the nonsmooth subproblems, one of them has a closed form solution through introducing appropriate proximal term, another is solved combining semi-smooth Newton (SSN) method. Based on these efficient strategies, we prove that our proposed FE-dABCD algorithm enjoys $O(\frac{1}{k^2})$ iteration complexity. Some numerical experiments are done and the numerical results show the efficiency of the FE-dABCD algorithm.
\end{abstract}

\section{Introduction}\label{Introduction}
Optimal control problems with constraints on the gradient of the state have a wide range of important applications. For instance, in cooling processes or structured optimization when high stresses have to be avoided. In cooling processes, constraints on the gradient of the state play an important role in practical applications where solidification of melts forms a critical process. In order to accelerate the production, it is highly desirable to speed up the cooling processes while avoiding damage of the products caused by large material stresses. Cooling frequently is described by systems of partial differential equations involving the temperature as a system variable, so that large (Von Mises) stresses in the optimization can be kept small by imposing bounds on the gradient of the temperature (see \cite{Wollner2013A, Nther2011Elliptic}).

Optimal control problems with constraints on the gradient of the state have caused much attention. For optimal control problems with pointwise constraints on the gradient of the state, there are some existing works on the optimal conditions \cite{Casas1993Optimal}, discretization \cite{Deckelnick2014A} and error analysis \cite{Wollner2013A, Nther2011Elliptic, Deckelnick2014A}. Since the Lagrange multiplier corresponding to the pointwise constraint on the gradient of the state in general only represents a regular Borel measure (see \cite{Casas1993Optimal}), the complementarity condition in the optimal conditions cannot be written into a pointwise form, which brings some difficulties and reduces flexibility of the numerical realization. As we know, although this difficulty can be solved by some common regularization approaches such as Lavrentiev regularization \cite{Meyer2006Optimal, Pr2007On, Chen2016A} and Moreau-Yosida regularization \cite{hintermuller2006feasible, hintermuller2006path, Itoa2003Semi, Kunisch2010Optimal}, the efficiency of these regularization approaches usually depends on the choice of the regularization parameters, which will also bring some difficulties. To relax the pointwise constraint, in this paper, as a model problem we consider the following elliptic PDE-constrained optimal control problem with box constraints on the control and integral constraint on the gradient of state.


\begin{equation}\label{original problem}
           \qquad \left\{ \begin{aligned}
        \min\limits_{(y,u)\in Y\times U}^{} \ \ J(y,u)&=\frac{1}{2}\|y-y_d\|_{L^2(\mathrm{\Omega})}^{2}+\frac{\alpha}{2}\|u\|_{L^2(\mathrm{\Omega})}^{2} \\
        {\mathrm{s.t.}}\qquad\quad\ \ \mathcal{A} y&=u+f\quad\mathrm{in}\  \mathrm{\Omega}, \\
         y&=0\qquad\quad\mathrm{on}\ \mathrm{\Gamma},\\
         \nabla y&\in \mathcal{K},\\
         u&\in \mathcal{W},
                          \end{aligned} \right.\tag{$\mathrm{P}$}
\end{equation}
where  $Y:=H_0^1(\mathrm{\Omega})$, $U:=L^2(\mathrm{\Omega})$, $\mathrm{\Omega}\subseteq \mathbb{R}^n\ (n=2,3)$ is a convex, open and bounded domain with $C^{1,1}$- or polygonal boundary $\mathrm{\Gamma}$; the desired state $y_d \in L^2(\mathrm{\Omega})$ and $f \in L^2(\mathrm{\Omega})$ are given; $a,\ b\in \mathbb{R}$ and $\alpha>0$ are given parameters; $\mathcal{K}=\{\mathbf{z}\in {L^{2}(\bar{\mathrm{\Omega}})}^{n}|\int_{\mathrm{\Omega}}|\mathbf{z}(x)|^2dx\leq\delta\}$ is a closed convex subset of ${L^{2}(\mathrm{\Omega})}^{n}$ with nonempty interior, $\mathcal{W}=\{v\in L^{\infty}(\mathrm{\Omega})|\ a\leq v \leq b\quad\mathrm{a.e.}\ x \in \mathrm{\Omega}\}$ is a nonempty convex closed subset of $U$; $\mathcal{A}$ is a uniformly elliptic operator
\begin{equation}\label{elliptic operator}
(\mathcal{A}y)(x):=-\sum^{n}_{i,j=1}\partial_{x_j}(a_{ij}y_{x_i})+c_0y,
\end{equation}
where $a_{ij}$, $c_0\in L^{\infty}(\mathrm{\Omega})$, $c_0\geq0$, $a_{ij}=a_{ji}$ and there is a constant $\theta>0$ such that
\begin{equation*}
\sum_{i,j=1}^{n}a_{ij}(x)\xi_i\xi_j\geq\theta\|\xi\|^2\quad \mathrm{for\ a.a.}\ x\in\mathrm{\Omega}\ \mathrm{and}\ \xi\in\mathbb{R}^n.
\end{equation*}
We use $|\cdot|$ to denote the Euclidean norm, use ${\rm{(\cdot,\cdot)}}$ to denote the inner product in $L^2(\mathrm{\Omega})$ and use $\|\cdot\|$ to denote the corresponding norm.

\begin{remark}
Although we assume that the operator $\mathcal{A}$ is a uniformly elliptic operator and Dirichlet boundary condition $y=0$ holds, we would like to point out that our considerations can also carry over to parabolic operators and more general boundary conditions of Robin type
\begin{equation*}
\partial_n y+\gamma y=g\qquad \mathrm{on}\ \mathrm{\Gamma},
\end{equation*}
where $g\in L^2(\mathrm{\Gamma})$ is given and $\gamma\in L^{\infty}(\mathrm{\Gamma})$ is a nonnegative coefficient.
\end{remark}

To numerically solve problem (\ref{original problem}), there are two possible ways. One is called {\emph{First discretize, then optimize}}, another is called {\emph{First optimize, then discretize}} \cite{Collis2002Analysis}. Independently of where discretization is located, the resulting finite dimensional equations are quite large. Hence, both cases require us to consider proposing an efficient algorithm based on the structure of the problem. In this paper, we use the {\emph{First discretize, then optimize}} approach. With respect to the discrete methods, we use the full discretization method, in which both the state and control are discretized by piecewise linear functions. 

As we know, there are many first order algorithms being used to solve finite dimensional large scale optimization fast, such as iterative soft thresholding algorithms (ISTA) \cite{Blumensath2008Iterative}, fast iterative soft thresholding algorithms (FISTA) \cite{Beck2009A}, accelerated proximal gradient (APG)-based method \cite{jiang2012inexact, toh2010accelerated} and alternating direction method of multipliers (ADMM) \cite{li2015qsdpnal, li2016schur, Chen2016A}. Motivated by the success of these first order algorithms, an APG method in function space (called Fast Inexact Proximal (FIP) method) was proposed to solve the elliptic optimal control problem involving $L^1$-control cost in \cite{schindele2016proximal}. It is known that whether the APG method is efficient depends closely on whether the step-length is close enough to the Lipschitz constant, however, the Lipschitz constant is not easy to estimate in usual, which largely limits the efficiency of APG method. Recently, an inexact heterogeneous ADMM (ihADMM) algorithm was proposed and used to solve optimal control problems with $L^{1}$-control cost in \cite{Song2017A}. Simultaneously, the authors also extended it to optimal control problems with $L^{2}$-control cost in \cite{Song2016A} and Lavrentiev-regularized state-constrained optimal control problem in \cite{Chen2016A}. It is known that its iteration scheme is simple and each subproblem can be solved efficiently. However, ihADMM algorithm only has $O(\frac{1}{k})$ iteration complexity.

Most of the papers mentioned above are devoted to solving the primal problem. However, Song et al. \cite{Song2018An} proposed an duality-based approach for PDE-constrained sparse optimization, which reformulated the problem as a multi-block unconstrained convex composite minimization problem and proposed a sGS-imABCD algorithm to solve the problem. Motivated by it, we consider solving (\ref{original problem}) via its dual. We find that we can also reformulate our problem as a multi-block unconstrained convex optimization problem and take advantage of the structure of it to construct an efficient algorithm to solve it efficiently and fast. Specifically, it is shown in Section \ref{A duality-based approach} that the dual problem of discretization version of (\ref{original problem}) can be written as
\begin{equation}\label{introduction: dual problem}
\begin{aligned}
\min\limits_{p,\lambda,\mu}\ F(p,\lambda,\mu)=&\frac{1}{2}\|{K_h}^Tp+\lambda-M_hy_d\|_{{M_h^{-1}}}^2+\frac{1}{2\alpha}\|M_hp-\mu\|_{{M_h^{-1}}}^2\\
&+\delta^*_{\mathcal{C}}(\lambda)+\delta^*_{\mathcal{S}}(\mu)+(M_hf,p)-\frac{1}{2}\|y_d\|_{M_h}^2,
\end{aligned}
\end{equation}
where $p$, $\lambda$, $\mu\in {\mathbb{R}}^n$ and for any given nonempty, closed convex subset $B$ of ${\mathbb{R}}^n$, $\delta_B(\cdot)$ denotes the indicator function of $B$. That is to say
\begin{equation}
 \delta_B(x)=\left\{ \begin{aligned}
         &1,\qquad &x\in B,\\
         &\infty,&x\notin B.
\end{aligned} \right.
\end{equation}
Based on inner product, the conjugate of $\delta_B(\cdot)$ is defined as follows
\begin{equation}
\begin{aligned}
\delta_B^*(y)&=\sup_{x}\ \{(y,x)-\delta_B(x)\},\\
&=\sup_{x\in B}\ (y,x).
\end{aligned}
\end{equation}

It is easy to see that (\ref{introduction: dual problem}) is a multi-block unconstrained minimization problem including two non-smooth terms, which are decoupled. In this case, block coordinate descent (BCD) method (see \cite{Grippo2000On, Sardy2000Block, Tseng1993Dual, Tseng2001Convergence}) is top-priority and appropriate. Combining BCD method and the acceleration technique in APG will result in accelerated block coordinate descent (ABCD) method. It is a natural idea that we use ABCD method to solve (\ref{introduction: dual problem}), however, the convergence property of $3$ block ABCD method can not be promised (see \cite{Chambolle2015A}). Fortunately, in our problem, $\lambda$ and $\mu$ can be seen as one block because of the important fact that $\lambda$ and $\mu$ are decoupled. Then (\ref{introduction: dual problem}) can be seen as an unconstrained convex optimization problem with coupled objective functions of the following form
\begin{equation}\label{Cui: multi-block}
\min_{v,w}\ f(v)+g(w)+\phi(v,w).
\end{equation}

For unconstrained convex optimization problems with coupled objective functions of form (\ref{Cui: multi-block}), an accelerated alternative descent (AAD) algorithm was proposed in \cite{Chambolle2015A} to solve it for the situation that the joint objective function $\phi$ is quadratic. However, the AAD method does not take the inexactness of the solutions of the associated subproblems into account. It is known that, in some case, exactly computing the solution of each subproblem is either impossible or extremely expensive. Sun et al. \cite{Sun2015An} proposed an inexact accelerated block descent (iABCD) method to solve least squares semidefinite programming (LSSDP) via its dual, whose basic idea is first reducing the two block nonsmooth terms into one through applying the Danskin-type theorem and then using APG method to solve the reduced problem. However, for the situation that the subproblem with respect to $w$ could not be solved exactly, Danskin-type theorem can no longer be used to reduce two block nonsmooth terms into one.

To overcome the bottlenecks above, an inexact majorized accelerated block coordinate descent (imABCD) method was proposed in {\cite[Chapter 3]{cui2016}}. Under suitable assumptions and certain inexactness criteria, the author proved that the imABCD method enjoys $O(\frac{1}{k^2})$ iteration complexity. Motivated by its success, in this paper, we propose a finite element duality-based inexact majorized accelerating block coordinate descent (FE-dABCD) algorithm. One distinctive feature of our proposed method is that it employs a majorization technique, which gives us a lot of freedoms to flexibly choose different proximal terms for different subproblems. Moreover, thanks to the inexactness of our proposed method, we have the essential flexibility that the inner subproblems are allowed to be solved only approximately. We would like to emphasize that these flexibilities are essential. First, the flexibility of choosing proximal terms makes each problem maintain good structure and can be solved efficiently. In addition, with some simple and implementable error tolerance criteria, the cost for inexactly solving the subproblems can be greatly reduced, which further contributes to the efficiency of the proposed method. Specifically, we can see from the content in Section \ref{An efficient ABCD algorithm} that the smooth subproblem, i.e. $p$-subproblem, is a $2*2$ block saddle point system, which can be solved efficiently by some Krylov-based methods with preconditioner, such as the generalized minimal residual (GMRES) method with preconditioner. As for the nonsmooth subproblems, i.e. subproblems with regard to $\lambda$ and $\mu$, the $\lambda$-subproblem has a closed form solution through introducing an appropriate proximal term and the $\mu$-subproblem is solved by combining semi-smooth Newton (SSN) method efficiently. Moreover, the $O(\frac{1}{k^2})$ iteration complexity of the FE-dABCD algorithm is also proved.

The rest of this article is structured as follows. In Section {\ref{Optimal Conditions}}, we will discuss the optimal conditions of problem (\ref{original problem}). Then we will consider its discretization in Section {\ref{Discretization}}. Section \ref{A FE-ABCD algorithm} will give a duality-based approach to transform the discretized problem into a multi-block unconstrained convex optimization problem. Then a brief sketch of the imABCD method will be given and our proposed FE-dABCD algorithm will be described in details. Some numerical experiments will be given in Section \ref{Numerical Experiment} to verify the efficiency of the proposed algorithm. Finally, we will make a simple summary in Section \ref{Conclusion}.

\section{Optimal Conditions}\label{Optimal Conditions}
For the existence and uniqueness of the solution of the PDE equation in problem (\ref{original problem})
\begin{equation}\label{pde equation}
 \left\{ \begin{aligned}
         \mathcal{A} y&=u+f\quad\mathrm{in}\  \mathrm{\Omega}, \\
         y&=0\qquad\quad\mathrm{on}\ \mathrm{\Gamma},
                          \end{aligned} \right.
\end{equation}
where $\mathcal{A}$ is defined by (\ref{elliptic operator}), the following theorem holds.
\begin{theorem}{\rm\cite[Theorem 1.23]{Hinze2009Optimization}}\label{Existence Theorem}
For each $u\in L^{2}(\mathrm{\Omega})$, there exists a unique weak solution $y_u\in H^{1}_{0}(\mathrm{\Omega})$ of {\rm(\ref{pde equation})} and satifies
\begin{equation*}
\|y_u\|_{H^1(\mathrm{\Omega})}\leq C_1(\|u\|_{L^2(\mathrm{\Omega})}+\|f\|_{L^2(\mathrm{\Omega})}),
\end{equation*}
where $C_1$ depends only on $a_{ij}$, $c_0$, $\mathrm{\Omega}$.
\end{theorem}

The weak formulation of (\ref{pde equation}) is given by
\begin{equation}
{\rm{Find}}\ y\in H^1_0(\mathrm{\Omega}):\quad a(y,v)=(u+f,v)_{L^2(\mathrm{\Omega})}\quad\forall v\in H^1_0(\mathrm{\Omega})
\end{equation}
with the bilinear form
\begin{equation}\label{bilinear form}
a(y,v)=\int_{\mathrm{\Omega}}\ \left(\sum^{n}_{i,j=1}a_{ij}y_{x_i}v_{x_j}+c_0yv\right)dx.
\end{equation}

The differentiability of the relation between the control and the state can be readily deduced from the implicit function theorem.

\begin{theorem}{\rm\cite[Theorem 2]{Casas1993Optimal}}\label{Differentiability Theorem}
The mapping $F:L^2(\mathrm{\Omega})\rightarrow H^1_0(\mathrm{\Omega})$ defined by $F(u)=y_u$ is of class $C^1$ and for every $u$, $v\in L^2(\mathrm{\Omega})$ the element $z=DF(u)\cdot v$ is the unique solution of the Dirichlet problem
\begin{equation}\label{state equation}
\left\{\begin{aligned}
   \mathcal{A} z&=v \qquad \mathrm{in}\  \mathrm{\Omega}, \\
   z&=0 \qquad \mathrm{on}\ \mathrm{\Gamma}.
\end{aligned}\right.
\end{equation}
\end{theorem}

Taking a minimizing sequence and arguing in the standard way, we obtain the existence of a solution for the optimal control problem (\ref{original problem}):

\begin{theorem}{\rm\cite[Theorem 1.43]{Hinze2009Optimization}}\label{Unique Theorem}
Assuming the existence of a feasible control {\rm(}i.e. a control $u\in \mathcal{W}$ such that $\nabla y\in \mathcal{K}${\rm)}, then problem {\rm(\ref{original problem})} has a unique solution.
\end{theorem}

To prove the optimality conditions for (\ref{original problem}), we first introduce the following theorem about the existence of Lagrange multiplier.
\begin{theorem}{\rm\cite[Theorem 5.2]{Casas1993BOUNDARY}}\label{Original Theorem}
Let $U$ and $Z$ be two Banach spaces and let $K\subset U$ and $C\subset Z$ be two convex subsets, $C$ having a nonempty interior. Let $\bar{u}\in K$ be a solution of the optimization problem:
\begin{equation*}
\left\{\begin{aligned}
&\min\ J(u)\\
&u\in K \quad {\rm{and}}\quad G(u)\in C,
\end{aligned}\right.
\end{equation*}
where $J:U\rightarrow(-\infty,+\infty]$ and $G:U\rightarrow Z$ are G{\rm{$\hat{a}$}}teaux differentiable at $\bar{u}$. Then there exist a real number $\bar{\lambda}\geq0$ and an element $\bar{\mu}\in Z'$ such that
\begin{subequations}
\begin{eqnarray}
&&\bar{\lambda}+\|\bar{\mu}\|_{Z^{'}}>0,\\\label{Original Theorem a}
&&\langle\bar{\mu},z-G(\bar{u})\rangle\ \leq\ 0,\quad \forall z\in C,\\\label{Original Theorem b}
&&\langle\bar{\lambda}J'(\bar{u})+[DG(\bar{u})]^{*}\bar{\mu},u-\bar{u}\rangle\ \geq0,\quad \forall u\in K.\label{Original Theorem c}
\end{eqnarray}
\end{subequations}
Moreover, $\bar{\lambda}$ can be taken equal to $1$ if the following condition of Slater type is satisfied:
\begin{equation}\label{Original slater condition}
\exists u_0\in K\quad such\ that\quad G(\bar{u})+DG(\bar{u})\cdot(u_0-\bar{u})\in C^{\circ}.
\end{equation}
\end{theorem}

The next theorem establishes the optimality conditions for (\ref{original problem}).
\begin{theorem}\label{Theorem for optimal conditions}
Let $\bar{u}$ be a solution of problem {\rm(\ref{original problem})}, then there exist a real number $\bar{\lambda}\geq 0$ and elements $\bar{y}\in H^1_0(\mathrm{\Omega})$, $\bar{p}\in H^1_0(\mathrm{\Omega})$ and $\bar{\mu}\in {L^2(\mathrm{\Omega})}^n$ satisfying
\begin{equation}\label{optimal condition 1}
\bar{\lambda}+\|\bar{\mu}\|_{{L^2(\mathrm{\Omega})}^n}>0,
\end{equation}

\begin{equation}\label{optimal condition 2}
\langle \mathcal{A}\bar{y},z \rangle=\langle \bar{u},z \rangle+\langle f,z \rangle,\quad \forall z\in H_0^1(\mathrm{\Omega}),
\end{equation}

\begin{equation}\label{optimal condition 3}
\langle \bar{p},\mathcal{A}z \rangle=\bar{\lambda}\langle \bar{y}-y_d,z \rangle+\langle \bar{\mu},\nabla z \rangle,\quad \forall z\in H_0^1(\mathrm{\Omega}),
\end{equation}

\begin{equation}\label{optimal condition 4}
\langle \bar{\mu},z-\nabla\bar{y} \rangle \leq0\qquad \forall z\in\mathcal{K},
\end{equation}

\begin{equation}\label{optimal condition 5}
\int_{\Omega}(\bar{p}+\bar{\lambda}\alpha\bar{u})(u-\bar{u})\ dx\geq0\qquad \forall u\in\mathcal{W}.
\end{equation}
Moreover, if the following Slater condition is verified
\begin{equation*}
\exists u_0\in \mathcal{W}\quad such\ that \quad (\nabla\bar{y}+\nabla z_0)\in \mathcal{K}^{\circ},
\end{equation*}
where $z_0$ is the solution of the following Dirichlet problem
\begin{equation*}
\left\{\begin{aligned}
   \mathcal{A} z_0&=u_0-\bar{u} \qquad \mathrm{in}\  \mathrm{\Omega}, \\
   z_0&=0 \qquad\qquad\  \mathrm{on}\ \mathrm{\Gamma},
\end{aligned}\right.
\end{equation*}
then the system (\ref{optimal condition 1})-(\ref{optimal condition 5}) is satisfied with $\bar{\lambda}=1$.
\end{theorem}
\begin{proof}
Applying Theorem \ref{Original Theorem} with $U$ as the control space, $Z={L^{2}(\mathrm{\Omega})}^{n}$, $J$ the functional to minimize, $G=\nabla\cdot F$, which is differential (Theorem \ref{Differentiability Theorem}), $K$ the convex set $\mathcal{W}$ of $U$ and $C$ the convex set $\mathcal{K}$ of ${L^{2}(\mathrm{\Omega})}^{n}$.

Then from (\ref{Original Theorem a}) and (\ref{Original Theorem b}), we deduce the existence of $\bar{\lambda}$ and $\bar{\mu}$ satisfying (\ref{optimal condition 1}) and (\ref{optimal condition 4}). Now we take $\bar{y}=y_{\bar{u}}$ and $\bar{p}$ the unique solution of (\ref{optimal condition 3}). Then it remains to prove inequality (\ref{optimal condition 5}), which is done by using the corresponding inequality (\ref{Original Theorem c}). $\forall v\in U$, let us take $z=DF(\bar{u})\cdot v\in H^{1}_{0}(\mathrm{\Omega})$ as a solution of
\begin{equation*}
\left\{\begin{aligned}
   \mathcal{A} z&=v \qquad \mathrm{in}\  \mathrm{\Omega}, \\
   z&=0 \qquad \mathrm{on}\ \mathrm{\Gamma}.
\end{aligned}\right.
\end{equation*}
Then we derive that
\begin{equation}
\begin{aligned}
&\ \bar{\lambda}J'(\bar{u})\cdot v+\langle[DG(\bar{u})]^{*}\bar{\mu},v\rangle\\
=&\ \bar{\lambda}\int_{\Omega}(\bar{y}-y_d)z\ dx+\bar{\lambda}\alpha\int_{\Omega}\bar{u}v\ dx+\langle\bar{\mu},DG(\bar{u})\cdot v\rangle,\\
=&\ \bar{\lambda}\int_{\Omega}(\bar{y}-y_d)z\ dx+\bar{\lambda}\alpha\int_{\Omega}\bar{u}v\ dx+\langle\bar{\mu},\nabla z\rangle,\\
=&\  \int_{\Omega}\bar{p}\mathcal{A} z\ dx+\bar{\lambda}\alpha\int_{\Omega}\bar{u}v\ dx,\\
=&\  \int_{\Omega}(\bar{p}+\bar{\lambda}\alpha\bar{u})v\ dx.
\end{aligned}
\end{equation}
Then from (\ref{Original Theorem c}), we get
\begin{equation}
\int_{\Omega}(\bar{p}+\bar{\lambda}\alpha\bar{u})(u-\bar{u})\ dx=\langle\bar{\lambda}J'(\bar{u})+[DG(\bar{u})]^{*}\bar{\mu},u-\bar{u}\rangle\ \geq0,\quad \forall u\in\mathcal{W}.
\end{equation}
Moreover, if there exists $u_0\in \mathcal{W}$ such that $(\nabla\bar{y}+\nabla z_0)\in \mathcal{K}^{\circ}$, where $z_0$ is the solution of the following Dirichlet problem
\begin{equation*}
\left\{\begin{aligned}
   \mathcal{A} z_0&=u_0-\bar{u} \qquad \mathrm{in}\  \mathrm{\Omega}, \\
   z_0&=0 \qquad\qquad\  \mathrm{on}\ \mathrm{\Gamma}.
\end{aligned}\right.
\end{equation*}
We know from Theorem \ref{Differentiability Theorem} that $z_0=DF(\bar{u})\cdot (u_0-\bar{u})$, then
\begin{equation*}
\nabla\bar{y}+\nabla z_0=\nabla\cdot F(\bar{u})+\nabla\cdot DF(\bar{u})\cdot (u_0-\bar{u})=G(\bar(u))+DG(\bar{u})\cdot (u_0-\bar{u})\in\mathcal{K}^{\circ},
\end{equation*}
which means the condition of Slater type (\ref{Original slater condition}) holds, then $\bar{\lambda}$ can be taken to 1.
$\Box$
\end{proof}

\section{Finite Element Discretization}\label{Discretization}
In order to tackle (\ref{original problem}) numerically, we discretize both the state $y$ and the control $u$ by continuous piecewise linear functions. Let us introduce a family of regular triangulations $\{T_h\}_{h>0}$ of $\mathrm{\Omega}$, i.e. $\bar{\mathrm{\Omega}}=\bigcup_{T\in T_h}\bar{T}$. With each element $T\in T_h$, we associate two parameters $\rho(T)$ and $R(T)$, where $\rho(T)$ denotes the diameter of the set $T$ and $R(T)$ is the diameter of the largest ball contained in $T$. The mesh size of $T_h$ is defined by $h=\max_{T\in T_h}\rho(T)$. We suppose the following standard assumption holds (see \cite{hinze2010variational}, \cite{Hinze2009Optimization}).
\begin{assumption}\label{assumption finite element}{(Regular and quasi-uniform triangulations)}
The domain $\mathrm{\Omega}$ is a open bounded and convex subset of $R^n$, $n=2,3$ and its boundary $\Gamma$ is a polygon ($n=2$) or a polyhedron (n=3). Moreover, there exist two positive constants $\rho$ and $R$ such that
\begin{equation*}
\frac{\rho(T)}{R(T)}\leq R,\quad \frac{h}{\rho(T)}\leq \rho
\end{equation*}
hold for all $T\in T_h$ and all $h>0$.  Let us define ${\bar{\mathrm{\Omega}}}_h=\bigcup_{T\in T_h}{T}$, and let $\mathrm{\Omega}_h\in\mathrm{\Omega}$ and $\mathrm{\Gamma}_h$ denote its interior and its boundary, respectively. In the case that $\mathrm{\Omega}$ has a $C^{1,1}$-boundary $\mathrm{\Gamma}$, we assume that $\bar{\mathrm{\Omega}}_h$ is a convex and that all boundary vertices of $\bar{\mathrm{\Omega}}_h$ are contained in $\mathrm{\Gamma}$, such that
\begin{equation*}
|\mathrm{\Omega}\backslash {\mathrm{\Omega}}_h|\leq ch^2,
\end{equation*}
where $|\cdot|$ denotes the measure of the set and $c>0$ is a constant.
\end{assumption}

Let $Z_h=\mathrm{span}\{\phi_{1},\phi_{2},...,\phi_{N_h}\}$ be the finite dimensional subspace, then
\begin{equation}
y_h(x)=\sum_{i=1}^{N_h}y_i\phi_i(x),\quad u_h(x)=\sum_{i=1}^{N_h}u_i\phi_i(x),
\end{equation}
\begin{equation}
\nabla y_h(x)=\sum_{i=1}^{N_h}y_i\nabla\phi_i(x)=\left[
\begin{array}{c}
\sum_{i=1}^{N_h}y_i\frac{\partial\phi_i(x)}{\partial x_1}\\
\vdots\\
\sum_{i=1}^{N_h}y_i\frac{\partial\phi_i(x)}{\partial x_n}
\end{array}
\right],
\end{equation}
\begin{equation}
\|\nabla y_h(x)\|^2=\sum_{j=1}^{n}\|\sum_{i=1}^{N_h}y_i\frac{\partial\phi_i(x)}{\partial x_j}\|^2=\sum_{j=1}^{n}y^TD_jy=\|y\|_{D_h}^2,
\end{equation}
where $y=(y_1,y_2,...,y_{N_h})^T$, $u=(u_1,u_2,...,u_{N_h})^T$, $D_j=\left(\int_\mathrm{{\Omega_h}}\frac{\partial\phi_i(x)}{\partial x_j}\cdot\frac{\partial\phi_l(x)}{\partial x_j}dx\right)_{i,l=1}^{N_h},\ j=1,\cdots,n$, $D_h=\sum_{j=1}^{n}D_j$. We define the following matrices
\begin{equation}
   {K_h}=\left(a(\phi_i,\phi_j)\right)_{i,j=1}^{N_h}\quad{\rm{and}}\quad {M_h}=\left(\int_{\mathrm{\Omega}_h}\phi_i\cdot\phi_j\ dx\right)_{i,j=1}^{N_h},
\end{equation}
where ${K_h}$ and ${M_h}$ denote the finite element stiffness matrix and mass matrix respectively. Let
\begin{equation}
  y_{d,h}(x)=\sum\limits_{i=1}^{N_h}y_d^i\phi_i(x),\quad f_{h}(x)=\sum\limits_{i=1}^{N_h}f_{i}\phi_i(x)
\end{equation}
be the nodal projection of $y_d(x)$, $f(x)$ onto $Z_h$, where $y_d^i=y_d(x^i)$, $f_{i}=f(x^i)$. Then the discretized problem can be rewritten into the following matrix-vector form
\begin{equation}\label{eqn:discretized problem matrix_vector}
           \qquad \left\{ \begin{aligned}
        \min\limits_{y,u\in \mathbb{R}^{N_h}}^{} \ \ J(y,u)&=\frac{1}{2}\|y-y_d\|_{M_h}^{2}+\frac{\alpha}{2}\|u\|_{M_h}^{2} \\
        {\mathrm{s.t.}}\qquad\ \ K_h y&=M_hu+M_hf\\
         y\in \mathcal{C}&=\{z\in{\mathbb{R}}^{N_h}|\ \|z\|^{2}_{D_h}\leq\delta\},\\
         u\in \mathcal{S}&=\{z\in{\mathbb{R}}^{N_h}|\ a\leq z\leq b\}.
                          \end{aligned} \right.\tag{$\mathrm{P}_h$}
\end{equation}

\section{A FE-dABCD algorithm}\label{A FE-ABCD algorithm}
In this part, we could see that the discretized version of original problem can be reformulated into a multi-block unconstrained convex optimization problem via its dual. Then we first focus on the inexact majorized accelerate block coordinate descent (imABCD) method which was proposed by Cui in {\cite[Chapter 3]{cui2016}} for a general class of problems and then explain how we extend it to our problem with some strategies according to the structure of the problem.
\subsection{A duality-based approach}\label{A duality-based approach}
We introduce two artificial variables $z$ and $w$, then (\ref{eqn:discretized problem matrix_vector}) can be transformed into the following problem
\begin{equation}\label{eqn:discretized problem zw}
           \qquad \left\{ \begin{aligned}
        \min\limits_{y,u,z,w\in \mathbb{R}^{N_h}}^{} \ \ J(y,u)&=\frac{1}{2}\|y-y_d\|_{M_h}^{2}+\frac{\alpha}{2}\|u\|_{M_h}^{2}+\delta_{\mathcal{C}}(z)+\delta_{\mathcal{S}}(w) \\
        {\mathrm{s.t.}}\qquad\quad\ K_h y&=M_hu+M_hf,\\
         y&=z,\\
         u&=w,
                          \end{aligned} \right.\tag{$\mathrm{\widetilde{P}}_h$}
\end{equation}
whose Lagrangian function is
\begin{equation}
\begin{aligned}
L(y,u,z,w;p,\lambda,\mu)=&\ \frac{1}{2}\|y-y_d\|_{M_h}^{2}+\frac{\alpha}{2}\|u\|_{M_h}^{2}+\delta_{\mathcal{C}}(z)+\delta_{\mathcal{S}}(w)\\
&+(K_hy-M_hu-M_hf,p)+(y-z,\lambda)+(u-w,\mu),\\
=&\ \underbrace{\frac{1}{2}\|y-y_d\|_{M_h}^{2}+(K_hy,p)+(y,\lambda)}_{L_1}\\
&+\underbrace{\frac{\alpha}{2}\|u\|_{M_h}^{2}-(M_hu,p)+(u,\mu)}_{L_2}\\
&+\underbrace{(-z,\lambda)+\delta_{\mathcal{C}}(z)}_{L_3}+\underbrace{(-w,\mu)+\delta_{\mathcal{S}}(w)}_{L_4}-(M_hf,p),
\end{aligned}
\end{equation}
where $p$, $\lambda$ and $\mu$ are Lagrangian multipliers associated with the three equality constraints respectively. Then the dual problem of (\ref{eqn:discretized problem zw}) is
\begin{equation}\label{original dual problem}
\max\limits_{p,\lambda,\mu}\ \min\limits_{y,u,z,w}\ L(y,u,z,w;p,\lambda,\mu)
\end{equation}
Let us focus on $\min\limits_{y,u,z,w}\ L(y,u,z,w;p,\lambda,\mu)$ first, we have
\begin{equation}\label{KKT1}
\left\{
\begin{aligned}
\nabla_y L_1\ =&\ M_h(y-y_d)+K_h^Tp+\lambda=0,\\
\nabla_u L_2\ =&\ \alpha M_hu-{M_h}^Tp+\mu=0,\\
\min_{z}\ L_3=&\ \min_{z}\ \{(-z,\lambda)+\delta_{\mathcal{C}}(z)\},\\
=&-\max_{z}\ \{(z,\lambda)-\delta_{\mathcal{C}}(z)\}=-\delta^*_{\mathcal{C}}(\lambda),\\
\min_{w}\ L_4=&\ \min_{w}\ \{(-w,\mu)+\delta_{\mathcal{S}}(w)\},\\
=&-\max_{w}\ \{(w,\mu)-\delta_{\mathcal{S}}(w)\}=-\delta^*_{\mathcal{S}}(\mu).
\end{aligned}\right.
\end{equation}
Above we use the concept of conjugate function. The conjugate function of $f$ is defined by
\begin{equation}\label{conjugate function of f}
f^*(x)=\sup_{y}\ \{(x,y)-f(y)\}.
\end{equation}
Insert (\ref{KKT1}) to $L(y,u,z,w;p,\lambda,\mu)$ we can get
\begin{equation*}
\begin{aligned}
&\min\limits_{y,u,z,w}\ L(y,u,z,w;p,\lambda,\mu)\\
=&\min\limits_{y,u,z,w}\ \frac{1}{2}\|M_h(y-y_d)+{K_h}^Tp+\lambda\|_{M_h^{-1}}^2-\frac{1}{2}\|{K_h}^Tp+\lambda-M_hy_d\|_{M_h^{-1}}^2+\frac{1}{2}\|y_d\|_{M_h}^2\\
&\qquad\quad+\frac{\alpha}{2}\|M_hu-\frac{1}{\alpha}M_hp+\frac{1}{\alpha}\mu\|_{M_h^{-1}}^2-\frac{1}{2\alpha}\|M_hp-\mu\|_{M_h^{-1}}^2\\
&\qquad\quad+\delta_{\mathcal{C}}(z)-(z,\lambda)+\delta_{\mathcal{S}}(w)-(w,\mu)-(M_hf,p),\\
=&-\frac{1}{2}\|{K_h}^Tp+\lambda-M_hy_d\|_{M_h^{-1}}^2-\frac{1}{2\alpha}\|M_hp-\mu\|_{M_h^{-1}}^2\\
&-\delta^*_{\mathcal{C}}(\lambda)-\delta^*_{\mathcal{S}}(\mu)-(M_hf,p)+\frac{1}{2}\|y_d\|_{M_h}^2.
\end{aligned}
\end{equation*}
Then (\ref{original dual problem}) is transformed into
\begin{equation}
\begin{aligned}
\max\limits_{p,\lambda,\mu}\ &-\frac{1}{2}\|{K_h}^Tp+\lambda-M_hy_d\|_{{M_h^{-1}}}^2-\frac{1}{2\alpha}\|M_hp-\mu\|_{M_h^{-1}}^2\\
&-\delta^*_{\mathcal{C}}(\lambda)-\delta^*_{\mathcal{S}}(\mu)-(M_hf,p)+\frac{1}{2}\|y_d\|_{M_h}^2,
\end{aligned}
\end{equation}
which is equivalent to
\begin{equation}\label{dual problem}
\begin{aligned}
\min\limits_{p,\lambda,\mu}\ F(p,\lambda,\mu)=&\frac{1}{2}\|{K_h}^Tp+\lambda-M_hy_d\|_{{M_h^{-1}}}^2+\frac{1}{2\alpha}\|M_hp-\mu\|_{{M_h^{-1}}}^2\\
&+\delta^*_{\mathcal{C}}(\lambda)+\delta^*_{\mathcal{S}}(\mu)+(M_hf,p)-\frac{1}{2}\|y_d\|_{M_h}^2,
\end{aligned}
\tag{$\mathrm{D\widetilde{P}}_h$}
\end{equation}
which is a multi-block unconstrained optimization problem. Thus, accelerated block coordinate descent (ABCD) method is preferred and appropriate. We will expend the imABCD algorithm, which was proposed in {\cite[Chapter 3]{cui2016}}, to our problem and employ some strategies according to the structure of our problem to solve it efficiently. We will give the details about the algorithm in the following part. 

\subsection{Inexact majorized ABCD algorithm}
It is well known that taking the inexactness of the solutions of associated subproblems into account is important for the numerical implementation. Thus, let us give a brief sketch of the imABCD method for a general class of unconstrained, multi-block convex optimization problems with coupled objective function
\begin{equation}\label{imABCD general model}
\min_{v,w}\ \theta(v,w):=f(v)+g(w)+\phi(v,w),
\end{equation}
where $f:\mathcal{V}\rightarrow(-\infty,+\infty]$ and $g:\mathcal{W}\rightarrow(-\infty,+\infty]$ are two convex functions (possibly nonsmooth), $\phi:\mathcal{V}\times\mathcal{W}\rightarrow(-\infty,+\infty]$ is a smooth convex function, and $\mathcal{V}$, $\mathcal{W}$ are real finite dimensional Hilbert spaces. To tackle with the general model (\ref{imABCD general model}), some more conditions and assumptions on $\phi$ are required.
\begin{assumption}\label{phi assumption}
The convex function $\phi:\mathcal{V}\times\mathcal{W}\rightarrow(-\infty,+\infty]$ is continuously differentiable with Lipschitz contunous gradient.
\end{assumption}

Let us denote $z:=(v,w)\in\mathcal{V}\times\mathcal{W}$. Hiriart-Urruty and Nguyen provided a second order Mean-Value Theorem (see {\cite[Theorem 2.3]{Hiriart1984Generalized}}) for $\phi$, which states that for any $z'$ and $z$ in $\mathcal{V}\times\mathcal{W}$, there exists $z''\in[z',z]$ and a self-adjoint positive semidefinite operator $\mathcal{G}\in{\partial}^{2}(z'')$ such that
\begin{equation}
\phi(z)=\phi(z')+\langle\nabla\phi(z'),z-z'\rangle+\frac{1}{2}\|z'-z\|_{\mathcal{G}}^{2},
\end{equation}
where ${\partial}^{2}(z'')$ denotes the Clarke's generalized Hessian at given $z''$ and $[z',z]$ denotes the line segment connecting $z'$ and $z$. Under Assumption \ref{phi assumption}, it is obvious that there exist two self-adjoint positive semidefinite linear operators $\mathcal{Q}$ and $\hat{\mathcal{Q}}:\mathcal{V}\times\mathcal{W}\rightarrow\mathcal{V}\times\mathcal{W}$ such that for any $z\in\mathcal{V}\times\mathcal{W}$,
\begin{equation}
\mathcal{Q}\preceq\mathcal{G}\preceq\hat{\mathcal{Q}},\qquad\forall\mathcal{G}\in{\partial}^{2}(z).
\end{equation}
Thus, for any $z$, $z'\in\mathcal{V}\times\mathcal{W}$, it holds
\begin{equation}
\phi(z)\geq\phi(z')+\langle\nabla\phi(z'),z-z'\rangle+\frac{1}{2}\|z'-z\|_{\mathcal{Q}}^{2}
\end{equation}
and
\begin{equation}
\phi(z)\leq\hat{\phi}(z;z'):=\phi(z')+\langle\nabla\phi(z'),z-z'\rangle+\frac{1}{2}\|z'-z\|_{\hat{\mathcal{Q}}}^{2}.
\end{equation}
Furthermore, we decompose the operators $\mathcal{Q}$ and $\hat{\mathcal{Q}}$ into the following block structures
\begin{equation}
\mathcal{Q}z:=\left(\begin{array}{cc}
\mathcal{Q}_{11}& \mathcal{Q}_{12}\\
\mathcal{Q}_{21}^* & \mathcal{Q}_{22}\\
\end{array}\right)
\left(\begin{array}{c}
v\\
w\\
\end{array}\right),\quad
\hat{\mathcal{Q}}z:=\left(\begin{array}{cc}
\hat{\mathcal{Q}}_{11}& \hat{\mathcal{Q}}_{12}\\
\hat{\mathcal{Q}}_{21}^* & \hat{\mathcal{Q}}_{22}\\
\end{array}\right)
\left(\begin{array}{c}
v\\
w\\
\end{array}\right),\quad
\forall z=(v,w)\in\mathcal{V}\times\mathcal{W}
\end{equation}
and assume $\mathcal{Q}$ and $\hat{\mathcal{Q}}$ satisfy the following conditions.
\begin{assumption}{\rm\cite[Assumption 3.1]{cui2016}}\label{Q assumption}
There exist two self-adjoint positive semidefinite linear operators $\mathcal{D}_{1}:\mathcal{V}\rightarrow\mathcal{V}$ and $\mathcal{D}_{2}:\mathcal{W}\rightarrow\mathcal{W}$ such that
\begin{equation}
\hat{\mathcal{Q}}:=\mathcal{Q}+{\rm{Diag}}(\mathcal{D}_{1},\mathcal{D}_2).
\end{equation}
Furthermore, $\hat{\mathcal{Q}}$ satisfies that $\hat{\mathcal{Q}}_{11}\succ0$ and $\hat{\mathcal{Q}}_{22}\succ0$.
\end{assumption}

Now we present the inexact majorized ABCD algorithm for the general problem (\ref{imABCD general model}) as follow.
\begin{algorithm}[H]
\caption{An inexact majorized ABCD algorithm for (\ref{imABCD general model})}
\label{algorithm_imABCD}
\textbf{Input}: $(v^{1},w^{1})=({\tilde{v}}^{0},{\tilde{w}}^{0})\in {\rm{dom}}(f)\times{\rm{dom}}(g)$. Let $\{\epsilon_{k}\}$ be a summable sequence of nonnegative numbers, and set $t_1=1$, $k=1$.\\
\textbf{Output}: $({\tilde{v}}^{k},{\tilde{w}}^{k})$\\
Iterative until convergence:
\begin{description}
\item[\textbf{Step 1}] Choose error tolerance $\delta^{k}_{v}\in\mathcal{V}$, $\delta^{k}_{w}\in\mathcal{W}$ such that
\begin{equation*}
\max{\{\delta^{k}_{v},\delta^{k}_{w}\}}\leq\epsilon_{k}.
\end{equation*}
\qquad Compute
\begin{equation*}
\left\{\begin{aligned}
{\tilde{v}}^k&=\arg\min_{v\in\mathcal{V}}\ \{f(v)+\hat{\phi}(v,w^k;v^k,w^k)-\langle\delta^{k}_{v},v\rangle\},\\
{\tilde{w}}^k&=\arg\min_{w\in\mathcal{W}}\ \{g(w)+\hat{\phi}({\tilde{v}}^k,w;v^k,w^k)-\langle\delta^{k}_{w},w\rangle\}.
\end{aligned}\right.
\end{equation*}
\item[\textbf{Step 2}] Set $t_{k+1}=\frac{1+\sqrt{1+4t_k^2}}{2}$ and $\beta=\frac{t_k-1}{t_{k+1}}$, compute
\begin{equation*}
v^{k+1}={\tilde{v}}^k+\beta_{k}({\tilde{v}}^k-{\tilde{v}}^{k-1}),\quad w^{k+1}={\tilde{w}}^k+\beta_{k}({\tilde{w}}^k-{\tilde{w}}^{k-1}).
\end{equation*}
\end{description}
\end{algorithm}

As for the convergence result of the imABCD algorithm, we can refer to the following theorem.
\begin{theorem}{\rm\cite[Theorem 3.2]{cui2016}}\label{cui convergence theorem}
Suppose that Assumption \ref{Q assumption} holds and the solution set $\mathrm{\Omega}$ of the problem (\ref{imABCD general model}) is non-empty. Let $z^*=(v^*,w^*)\in\mathrm{\Omega}$. Assume that $\sum\limits^{\infty}_{k=1}k\epsilon_k<\infty$. Then the sequence $\{{\tilde{z}}^k\}:=\{({\tilde{v}}^k,{\tilde{w}}^k)\}$ generated by the Algorithm \ref{algorithm_imABCD} satisfies that
\begin{equation*}
\theta({\tilde{z}}^k)-\theta(z^*)\leq\frac{2\|{\tilde{z}}^0-z^*\|_{\mathscr{S}}^{2}+c_0}{(k+1)^2},\qquad\forall k\geq1,
\end{equation*}
where $c_0$ is a constant number and $\mathscr{S}:={\rm{Diag}}(\mathcal{D}_1,\mathcal{D}_2+\mathcal{Q}_{22})$.
\end{theorem}

\subsection{An FE-dABCD algorithm for (\ref{dual problem})}\label{An efficient ABCD algorithm}
Although at first glance, (\ref{dual problem}) is a $3$ block unconstrained convex optimization problem, fortunately, $\lambda$ and $\mu$ can be seen as one block because of the important fact that $\lambda$ and $\mu$ are decoupled. Seeing $p$ as $v$, $(\lambda,\mu)$ as $w$ and regarding $(M_hf,p)$ as $f(v)$, $\delta^*_{\mathcal{C}}(\lambda)+\delta^*_{\mathcal{S}}(\mu)$ as $g(w)$, $\frac{1}{2}\|{K_h}^Tp+\lambda-M_hy_d\|_{{M_h^{-1}}}^2+\frac{1}{2\alpha}\|M_hp-\mu\|_{M_h^{-1}}^2-\frac{1}{2}\|y_d\|_{M_h}^2$ as $\phi(v,w)$ in (\ref{imABCD general model}) and Algorithm \ref{algorithm_imABCD}, we first focus on applying imABCD algorithm to (\ref{dual problem}). Since $\phi$ is quadratic, we can take
\begin{equation}
\mathcal{Q}=\left[\begin{array}{ccccc}
K_h{M_h}^{-1}K_h^{T}+\frac{1}{\alpha}M_h &\ \ & K_hM_h^{-1}&\ \ &-\frac{1}{\alpha}I\\
&&&&\\
M_h^{-1}K_h^{T} && M_h^{-1} && 0\\
&&&&\\
-\frac{1}{\alpha}I && 0 && \frac{1}{\alpha}M_h^{-1}\\
\end{array}\right],
\end{equation}
where
\begin{equation}
\mathcal{Q}_{11}=K_h{M_h}^{-1}K_h^{T}+\frac{1}{\alpha}M_h,\quad
\mathcal{Q}_{22}=\left[\begin{array}{ccc}
M_h^{-1} &\ \ &0\\
0 && \frac{1}{\alpha}M_h^{-1}\\
\end{array}\right].
\end{equation}
In addition, we assume that there exist two self-adjoint positive semidefinite operators $\mathcal{D}_1$ and $\mathcal{D}_2$ such that Assumption \ref{Q assumption} holds, which implies that we should majorize $\phi(p,\lambda,\mu)$ at $z'=(p',\lambda',\mu')$ as
\begin{equation}
\phi(z)\leq\hat{\phi}(z;z')=\phi(z)+\frac{1}{2}\|p-p'\|_{\mathcal{D}_1}^{2}+\frac{1}{2}\left\Arrowvert
\left(\begin{array}{c}
\lambda\\
\mu\\
\end{array}\right)-\left(\begin{array}{c}
\lambda'\\
\mu'\\
\end{array}\right)\right\Arrowvert_{\mathcal{D}_2}^{2}.
\end{equation}
Based on the content above, we give the framework of imABCD for (\ref{dual problem}).
\begin{algorithm}[H]\label{algo: imABCD for DPh}
\caption{An inexact majorized ABCD algorithm for (\ref{dual problem})}
\textbf{Input}: $(p^{1},\lambda^{1},\mu^{1})=({\tilde{p}}^{0},{\tilde{\lambda}}^{0},{\tilde{\mu}}^{0})\in {\mathbb{R}}^{N_h}\times{\rm{dom}}(\delta^*_{\mathcal{C}})\times{\rm{dom}}(\delta^*_{\mathcal{S}})$. Let $\{\epsilon_{k}\}$ be a summable sequence of nonnegative numbers, and set $t_1=1$, $k=1$.\\
\textbf{Output}: $({\tilde{p}}^{k},{\tilde{\lambda}}^{k},{\tilde{\mu}}^{k})$\\
Iterative until convergence:
\begin{description}
\item[\textbf{Step 1}] Choose error tolerance $\delta^{k}_{p}$, $\delta^{k}_{\lambda}$ and $\delta^{k}_{\mu}$ such that
\begin{equation*}
\max{\{\delta^{k}_{p},\delta^{k}_{\lambda},\delta^{k}_{\mu}\}}\leq\epsilon_{k}.
\end{equation*}
\qquad Compute
\begin{equation*}
\begin{aligned}
{\tilde{p}}^k&=\arg\min_{p}\ \left\{(M_hf,p)+{\phi}(p,\lambda^k,\mu^k)+\frac{1}{2}\|p-p^k\|_{\mathcal{D}_1}^{2}-\langle\delta^{k}_{p},p\rangle\right\},\\
({\tilde{\lambda}}^k,{\tilde{\mu}}^k)&=\arg\min_{(\lambda,\mu)}\ \left\{\delta^*_{\mathcal{C}}(\lambda)+\delta^*_{\mathcal{S}}(\mu)+{\phi}({\tilde{p}}^k,\lambda,\mu)\right.\\
&\qquad\qquad\quad+\left.\frac{1}{2}\left\Arrowvert
\left(\begin{array}{c}
\lambda\\
\mu\\
\end{array}\right)-\left(\begin{array}{c}
\lambda^k\\
\mu^k\\
\end{array}\right)\right\Arrowvert_{\mathcal{D}_2}^{2}-\langle\delta^{k}_{\lambda},\lambda\rangle-\langle\delta^{k}_{\mu},\mu\rangle\right\}
\end{aligned}
\end{equation*}
\item[\textbf{Step 2}] Set $t_{k+1}=\frac{1+\sqrt{1+4t_k^2}}{2}$ and $\beta=\frac{t_k-1}{t_{k+1}}$, compute
\begin{equation*}
p^{k+1}={\tilde{p}}^k+\beta_{k}({\tilde{p}}^k-{\tilde{p}}^{k-1}),\quad {\lambda}^{k+1}={\tilde{{\lambda}}}^k+\beta_{k}({\tilde{{\lambda}}}^k-{\tilde{{\lambda}}}^{k-1}),\quad {\mu}^{k+1}={\tilde{{\mu}}}^k+\beta_{k}({\tilde{{\mu}}}^k-{\tilde{{\mu}}}^{k-1}).
\end{equation*}
\end{description}
\end{algorithm}

As we know, appropriate operators $\mathcal{D}_1$ and $\mathcal{D}_2$ are important for both the theory analysis and numerical implementation. Thus what we concern most now is how to choose the operators $\mathcal{D}_1$ and $\mathcal{D}_2$. In the view of numerical efficiency, the general principle is that in the premise of Assumption \ref{Q assumption}, both $\mathcal{D}_1$ and $\mathcal{D}_2$ should be as small as possible to get larger step-lengths and make the corresponding subproblems easy to solve.

First, let us focus on the choice for operator $\mathcal{D}_1$. Without the proximal term $\frac{1}{2}\|p-p'\|_{\mathcal{D}_1}^{2}$ and the error term $\langle\delta^{k}_{p},p\rangle$, let $F_1(p,\lambda^k,\mu^k)=(M_hf,p)+{\phi}(p,\lambda^k,\mu^k)$, then
\begin{equation}\label{KKT2:sub1}
\begin{aligned}
&\nabla_pF_1({\tilde{p}}^k,{\lambda}^k,{\mu}^k)\\
=&K_h{M_h^{-1}}({K_h}^T{\tilde{p}}^k-M_hy_d+{\lambda}^k)+\frac{1}{\alpha}(M_h{\tilde{p}}^k-{\mu}^k)+M_hf=0.
\end{aligned}
\end{equation}
Combining (\ref{KKT1}) and (\ref{KKT2:sub1}) we can derive that
\begin{equation}
\left\{
\begin{aligned}
-K_h{\tilde{y}}^k+\frac{1}{\alpha}(M_h{\tilde{p}}^k-\mu^{k})+M_hf&=0,\\
M_h({\tilde{y}}^k-y_d)+K_h^{T}{\tilde{p}}^k+\lambda^{k}&=0.
\end{aligned}\right.
\end{equation}
\begin{equation}\label{al:step1 equation system}
\Leftrightarrow\left[\begin{array}{cc}
\frac{1}{\alpha}M_h & -K_h\\
{K_h}^T & M_h\\
\end{array}\right]
\left[\begin{array}{c}
{\tilde{p}}^k\\
{\tilde{y}}^k\\
\end{array}\right]=
\left[\begin{array}{c}
\frac{1}{\alpha}\mu^{k}-M_hf\\
M_hy_d-\lambda^{k}\\
\end{array}\right].
\end{equation}
As we see, $p$-subproblem can be transformed into a $2*2$ block saddle point linear system, which can be solved by GMRES with preconditioner efficiently. So we only need set $\mathcal{D}_1=0$.

Next, for the choice of operator $\mathcal{D}_2$, the following facts about proximal operator
\begin{equation}\label{proximal operator}
\mathrm{prox}_{f}(x)=\arg\min_{y}\ \left\{f(y)+\frac{1}{2}\|y-x\|^2\right\},
\end{equation}
will be used in the following. For proximal operator (\ref{proximal operator}), there hold
\begin{equation}
\mathrm{prox}_{f}(x)+\mathrm{prox}_{f^*}(x)=x
\end{equation}
and 
\begin{equation}
\mathrm{prox}_{\frac{f}{\sigma}}(x)+\frac{1}{\sigma}\mathrm{prox}_{\sigma f^*}(\sigma x)=x,
\end{equation}
where $f^*$ is the conjugate function of $f$ defined as (\ref{conjugate function of f}).

Without the proximal term ${{\frac{1}{2}}}\left\Arrowvert
\left(\begin{array}{c}
\lambda\\
\mu\\
\end{array}\right)-\left(\begin{array}{c}
\lambda^k\\
\mu^k\\
\end{array}\right)\right\Arrowvert_{\mathcal{D}_2}^{2}$ and the error terms $\langle\delta^{k}_{\lambda},\lambda\rangle$ and $\langle\delta^{k}_{\mu},\mu\rangle$, the subproblem with regard to $(\lambda,\mu)$ will be
\begin{equation}\label{lambda mu subproblem}
\begin{aligned}
({\tilde{\lambda}}^k,{\tilde{\mu}}^k)=\arg\min_{(\lambda,\mu)}\ &\left\{\delta^*_{\mathcal{C}}(\lambda)+\delta^*_{\mathcal{S}}(\mu)+\frac{1}{2\alpha}\|M_h{\tilde{p}}^k-\mu\|_{M_h^{-1}}^2\right.\\
&\ \left.+\frac{1}{2}\|{K_h}^T{\tilde{p}}^k+\lambda-M_hy_d\|_{{M_h^{-1}}}^2\right\}.
\end{aligned}
\end{equation}
It is not difficult to see from (\ref{lambda mu subproblem}) that $\lambda$ and $\mu$ are decoupled, which means we can compute ${\tilde{\lambda}}^k$ and ${\tilde{\mu}}^k$  through the following two optimization problems respectively
\begin{equation}\label{lambda subproblem}
\begin{aligned}
{\tilde{\lambda}}^k=&\arg\min_{\lambda}\ \left\{\delta^*_{\mathcal{C}}(\lambda)+\frac{1}{2}\|{K_h}^T{\tilde{p}}^k+\lambda-M_hy_d\|_{{M_h^{-1}}}^2\right\},\\
=&\arg\min_{\lambda}\ \left\{\delta^*_{\mathcal{C}}(\lambda)+\frac{1}{2}\|\lambda-g^k\|_{{M_h^{-1}}}^2\right\},\\
\end{aligned}
\end{equation}
and 
\begin{equation}\label{mu subproblem}
{\tilde{\mu}}^k=\arg\min_{\mu}\ \left\{\delta^*_{\mathcal{S}}(\mu)+\frac{1}{2\alpha}\|{M_h\tilde{p}}^k-\mu\|_{M_h^{-1}}^2\right\},
\end{equation}
where $g^k=M_hy_d-{K_h}^T{\tilde{p}}^k$.

To make (\ref{lambda subproblem}) have a closed form solution, a natural choice is to add a proximal term $\frac{1}{2}\|{\lambda}-{\lambda}^k\|_{\sigma I-M_h^{-1}}^2$, where $\sigma$ is chosen such that $\sigma I-M_h^{-1}$ is a positive semidefinite matrix. Then,
\begin{equation}\label{lambdatilde solution}
\begin{aligned}
{\tilde{\lambda}}^k=&\arg\min_{\lambda}\ \left\{\delta^*_{\mathcal{C}}(\lambda)+\frac{1}{2}\|\lambda-g^k\|_{{M_h^{-1}}}^2+\frac{1}{2}\|{\lambda}-{\lambda}^k\|_{\sigma I-M_h^{-1}}^2\right\},\\
=&\arg\min_{{\lambda}}\ \left\{\delta^*_{\mathcal{C}}(\lambda)+\frac{\sigma}{2}{\lambda}^T{\lambda}-{\lambda}^T({M_h^{-1}}{g^k}+(\sigma I-M_h^{-1}){{{\lambda}}^k})\right\},\\
=&\arg\min_{\lambda}\ \left\{\frac{1}{\sigma}\delta^*_{\mathcal{C}}(\lambda)+\frac{1}{2}\|\lambda-d^k\|^2\right\},\\
=&\ \mathrm{prox}_{\frac{\delta^*_{\mathcal{C}}}{\sigma}}(d^k),\\
=&\ d^k-\frac{1}{\sigma}\mathrm{prox}_{\sigma\delta_{\mathcal{C}}}(\sigma d^k).
\end{aligned}
\end{equation}
In the last formula in (\ref{lambdatilde solution}), $d^k=\frac{1}{\sigma}y_d+{\lambda}^k-\frac{1}{\sigma}M_h^{-1}(K_h^T{\tilde{p}}^k+{\lambda}^k)$, which can be get by solving the following linear system
\begin{equation}
M_hd^k=\frac{1}{\sigma}M_hy_d+M_h{\lambda}^k-\frac{1}{\sigma}(K_h^T{\tilde{p}}^k+{\lambda}^k)
\end{equation}
and $\mathrm{prox}_{\sigma\delta_{\mathcal{C}}}(\sigma d^k)$ can be computed as follows
\begin{equation}
\begin{aligned}
\mathrm{prox}_{\sigma\delta_{\mathcal{C}}}(\sigma d^k)&=\arg\min_{z}\ \left\{\sigma\delta_{\mathcal{C}}(z)+\frac{1}{2}\|z-\sigma d^k\|^2\right\},\\
&=\arg\min_{z\in\mathcal{C}}\ \frac{1}{2\sigma}\|z-\sigma d^k\|^2,\\
&=\mathrm{\Pi}_{\mathcal{C}}(\sigma d^k),
\end{aligned}
\end{equation}
where $\mathrm{\Pi}_{\mathcal{C}}$ denotes the projection operator to set $\mathcal{C}$. So
\begin{equation}\label{lambda}
\tilde{\lambda}^{k}=d^k-\frac{1}{\sigma}\mathrm{\Pi}_{\mathcal{C}}(\sigma d^k).
\end{equation}
For the computation of projection to set $\mathcal{C}$, please refer to Subsection \ref{projection}. 

Here we would like to point out that to make $\sigma I-M_h^{-1}$ positive semidefinite, $\sigma$ has to be not less than the biggest eigenvalue of $M_h^{-1}$. However in practice, the biggest eigenvalue of $M_h^{-1}$ will not be easy to compute when $h$ is very small, which means the choice of appropriate $\sigma$ is not easy. To avoid the computation of the biggest eigenvalue, we utilize the lump mass matrix $W_h$ defined by
\begin{equation}\label{Wh}
  W_h={\rm{diag}}\left(\int_{\mathrm{\Omega}_h}\ \phi_i(x)\ dx\right)^{N_h}_{i=1},
\end{equation}
which is a diagonal matrix. For the mass matrix $M_h$ and the lump mass matrix $W_h$, the following proposition hold.
\begin{proposition}{\rm\cite[Table 1]{Wathen1987Realistic}}\label{eqn:martix properties}
$\forall$ $z\in \mathbb{R}^{N_h}$, the following inequalities hold:
\begin{equation*}
  \|z\|^2_{M_h}\leq\|z\|^2_{W_h}\leq c_n\|z\|^2_{M_h},\quad where\quad c_n=\left\{
  \begin{aligned}
  &4\quad if\quad n=2,\\
  &5\quad if\quad n=3.
  \end{aligned}\right.
\end{equation*}
\end{proposition}
We can derive from Proposition \ref{eqn:martix properties} that $\forall$ $z\in \mathbb{R}^{N_h}$,
\begin{equation*}
  \|z\|^2_{W_h^{-1}}\leq\|z\|^2_{M_h^{-1}}\leq c_n\|z\|^2_{W_h^{-1}},\quad where\quad c_n=\left\{
  \begin{aligned}
  &4\quad if\quad n=2,\\
  &5\quad if\quad n=3.
  \end{aligned}\right.
\end{equation*}
It is clear that $c_nW_h^{-1}-{{M_h^{-1}}}$ is positive semidefinite. Also, in fact, each principle diagonal element of $W_h$ is twice as the counterpart of $M_h$, so $W_h$ is a diagonal matrix with positive principle diagonal elements. Let $\omega_m$ be the inverse of the smallest principle diagonal element of $W_h$, then we can set $\sigma=c_n\omega_m$. 

To make (\ref{mu subproblem}) have a closed form solution, similar to the discussion about the $\lambda$-subproblem, setting $\frac{1}{2\alpha}\|\mu-\mu^k\|_{\sigma I-{{M_h^{-1}}}}^2$, where $\sigma=c_n\omega_m$, as the proximal term is a natural choice. However, we say $\frac{1}{2\alpha}\|\mu-\mu^k\|_{c_nW_h^{-1}-{{M_h^{-1}}}}^2$ is a better proximal term for this subproblem because of the fact that $c_nW_h^{-1}-{{M_h^{-1}}}\prec\sigma I-{{M_h^{-1}}}$ and $\mathcal{S}$ is a box set. Then,
\begin{equation}
\begin{aligned}
{\tilde{\mu}}^k=&\arg\min_{\mu}\ \left\{\delta^*_{\mathcal{S}}(\mu)+\frac{1}{2\alpha}\|\mu-M_h\tilde{p}^k\|_{{M_h^{-1}}}^2+\frac{1}{2\alpha}\|\mu-\mu^k\|_{c_n W_h^{-1}-{{M_h^{-1}}}}^2\right\},\\
=&\arg\min_{\mu}\ \left\{\alpha\delta^*_{\mathcal{S}}(\mu)+\frac{1}{2}\|\mu-M_h\tilde{p}^k\|_{{M_h^{-1}}}^2+\frac{1}{2}\|\mu-\mu^k\|_{c_n W_h^{-1}-{{M_h^{-1}}}}^2\right\},\\
=&\arg\min_{\mu}\ \left\{\alpha\delta^*_{\mathcal{S}}(\mu)+\frac{c_n}{2}{\mu}^TW_h^{-1}{\mu}-{\mu}^T(\tilde{p}^k+(c_n W_h^{-1}-M_h^{-1}){\mu}^k)\right\},\\
=&\frac{1}{c}W_h\left(q^k-\alpha\mathrm{\Pi}_{\mathcal{S}}(\frac{1}{\alpha}q^k)\right),
\end{aligned}
\end{equation}
where $q^k=\tilde{p}^k+(c W_h^{-1}-M_h^{-1}){\mu}^k=\tilde{p}^k+c_n W_h^{-1}{\mu}^k-s^k$ and $s^k$ is the solution of the following linear system
\begin{equation}
M_hs^k={\mu}^k.
\end{equation}

\begin{remark}
We would like to emphasize here that for the $\lambda$-subproblem, we do not use $\frac{1}{2}\|\lambda-\lambda^k\|_{c_nW_h^{-1}-{{M_h^{-1}}}}^2$ as the proximal term because $\mathcal{C}$ is not a box set. If we do so, it will make the $\lambda$-subproblem do not have a closed form solution.
\end{remark}

Based on the content above, we can see that
\begin{equation}
\mathcal{D}_{2}=\left[\begin{array}{ccc}
\sigma I-M_h^{-1} &\ \ &0\\
0 && \frac{1}{\alpha}c_nW_h^{-1}-\frac{1}{\alpha}M_h^{-1}\\
\end{array}\right].
\end{equation}
Then we give the detailed framework of our inexact duality based majorized ABCD for (\ref{dual problem}) as follows
\begin{algorithm}[H]
\caption{An FE-dABCD algorithm for (\ref{dual problem})}
\label{DimADMM}
\textbf{Input}: $(p^{1},\lambda^{1},\mu^{1})=({\tilde{p}}^{0},{\tilde{\lambda}}^{0},{\tilde{\mu}}^{0})\in {\mathbb{R}}^{N_h}\times{\rm{dom}}(\delta^*_{\mathcal{C}})\times{\rm{dom}}(\delta^*_{\mathcal{S}})$. Let $\{\epsilon_{k}\}$ be a sequence of nonnegative numbers such that $\sum\limits^{\infty}_{k=1}k\epsilon_k<\infty$. Set $t_1=1$, $k=1$.\\
\textbf{Output}: $({\tilde{p}}^{k},{\tilde{\lambda}}^{k},{\tilde{\mu}}^{k})$\\
Iterative until convergence:
\begin{description}
\item[\textbf{Step 1}] Choose error tolerance $\delta^{k}_{p}$, $\delta^{k}_{\lambda}$ and $\delta^{k}_{\mu}$ such that
\begin{equation*}
\max{\{\delta^{k}_{p},\delta^{k}_{\lambda},\delta^{k}_{\mu}\}}\leq\epsilon_{k}.
\end{equation*}
\qquad Compute
\begin{equation*}
\begin{aligned}
{\tilde{p}}^k&=\arg\min_{p}\ \left\{(M_hf,p)+{\phi}(p,\lambda^k,\mu^k)-\langle\delta^{k}_{p},p\rangle\right\},\\
{\tilde{\lambda}}^k&=\arg\min_{\lambda}\ \left\{\delta^*_{\mathcal{C}}(\lambda)+\frac{1}{2}\|\lambda-M_hy_d+{K_h}^T{\tilde{p}}^k\|_{{M_h^{-1}}}^2+\frac{1}{2}\|{\lambda}-{\lambda}^k\|_{\sigma I-M_h^{-1}}^2-\langle\delta^{k}_{\lambda},\lambda\rangle\right\},\\
{\tilde{\mu}}^k&=\arg\min_{\mu}\ \left\{\delta^*_{\mathcal{S}}(\mu)+\frac{1}{2\alpha}\|\mu-M_h\tilde{p}^k\|_{{M_h^{-1}}}^2+\frac{1}{2\alpha}\|\mu-\mu^k\|_{c_n W_h^{-1}-{{M_h^{-1}}}}^2-\langle\delta^{k}_{\mu},\mu\rangle\right\}.
\end{aligned}
\end{equation*}
\item[\textbf{Step 2}] Set $t_{k+1}=\frac{1+\sqrt{1+4t_k^2}}{2}$ and $\beta=\frac{t_k-1}{t_{k+1}}$, compute
\begin{equation*}
p^{k+1}={\tilde{p}}^k+\beta_{k}({\tilde{p}}^k-{\tilde{p}}^{k-1}),\quad{\lambda}^{k+1}={\tilde{{\lambda}}}^k+\beta_{k}({\tilde{{\lambda}}}^k-{\tilde{{\lambda}}}^{k-1}),\quad {\mu}^{k+1}={\tilde{{\mu}}}^k+\beta_{k}({\tilde{{\mu}}}^k-{\tilde{{\mu}}}^{k-1}).
\end{equation*}
\end{description}
\end{algorithm}

We can show our FE-dABCD algorithm also has $O(\frac{1}{k^2})$ iteration complexity based on Theorem \ref{cui convergence theorem}.
\begin{theorem}\label{FE-dABCD iteration complexity}
Assume that $\sum\limits^{\infty}_{k=1}k\epsilon_k<\infty$. Let $\{{\tilde{z}}^k\}:=\{({\tilde{p}}^k,{\tilde{\lambda}}^k,{\tilde{\mu}}^k)\}$ be the sequence generated by Algorithm \ref{DimADMM}. Then we have
\begin{equation*}
F({\tilde{z}}^k)-F(z^*)\leq\frac{2\|{\tilde{z}}^0-z^*\|_{\mathscr{S}}^{2}+c_0}{(k+1)^2},\qquad\forall k\geq1,
\end{equation*}
where $c_0$ is a constant number, $\mathscr{S}:={\rm{Diag}}(\mathcal{D}_1,\mathcal{D}_2+\mathcal{Q}_{22})$ and $F(\cdot)$ is the objective function of problem {\rm(\ref{dual problem})}.
\end{theorem}
\begin{proof}
Based on Theorem \ref{cui convergence theorem}, what we have to do is just to verify that Assumption \ref{Q assumption} holds. Recall that 
\begin{equation}
\mathcal{D}_{1}=0,\quad
\mathcal{D}_{2}=\left[\begin{array}{ccc}
\sigma I-M_h^{-1} &\ \ &0\\
0 && \frac{1}{\alpha}c_nW_h^{-1}-\frac{1}{\alpha}M_h^{-1}\\
\end{array}\right],
\end{equation}
and 
\begin{equation}
\mathcal{Q}=\left[\begin{array}{ccccc}
K_h{M_h}^{-1}K_h^{T}+\frac{1}{\alpha}M_h &\ \ & K_hM_h^{-1}&\ \ &-\frac{1}{\alpha}I\\
&&&&\\
M_h^{-1}K_h^{T} && M_h^{-1} && 0\\
&&&&\\
-\frac{1}{\alpha}I && 0 && \frac{1}{\alpha}M_h^{-1}\\
\end{array}\right].
\end{equation}
Then we have
\begin{equation}
\hat{\mathcal{Q}}=\mathcal{Q}+{\rm{Diag}}(\mathcal{D}_{1},\mathcal{D}_2)=
\left[\begin{array}{ccccc}
K_h{M_h}^{-1}K_h^{T}+\frac{1}{\alpha}M_h &\ \ & K_hM_h^{-1}&\ \ &-\frac{1}{\alpha}I\\
&&&&\\
M_h^{-1}K_h^{T} && \sigma I && 0\\
&&&&\\
-\frac{1}{\alpha}I && 0 && \frac{1}{\alpha}c_nW_h^{-1}\\
\end{array}\right].
\end{equation}
That is to say
\begin{equation}
\hat{\mathcal{Q}}_{11}=K_h{M_h}^{-1}K_h^{T}+\frac{1}{\alpha}M_h,\quad
\hat{\mathcal{Q}}_{22}=\left[\begin{array}{ccc}
\sigma I &\ \ &0\\
0 && \frac{1}{\alpha}c_nW_h^{-1}\\
\end{array}\right].
\end{equation}
Since stiffness matrix $K_h$ and mass matrix $M_h$ are both symmetric positive definite matrices. Moreover, from Proposition \ref{eqn:martix properties}, we know that $\hat{\mathcal{Q}}_{11}\succ0$ and $\hat{\mathcal{Q}}_{22}\succ0$ hold. Thus we can establish the convergence of Algorithm \ref{DimADMM}.
$\Box$
\end{proof}

Assume that the sequence $\{({\tilde{p}}^k,{\tilde{\lambda}}^k,{\tilde{\mu}}^k)\}$ generated by Algorithm \ref{DimADMM} converges to $(p^*,\lambda^*,\mu^*)$, then from (\ref{KKT1}) we can derive that
\begin{equation}
M_hu^*=\frac{1}{\alpha}(M_hp^*-\mu^*).
\end{equation}

\subsection{An efficient preconditioner for the $p$-subproblem}
As we said in Section \ref{An efficient ABCD algorithm}, the $p$-subproblem can be transformed into a equation system (\ref{al:step1 equation system}), which is a special case of the generalized saddle-point system and can be inexactly solved by Krylov-based methods with preconditioner. It is easy to see that (\ref{al:step1 equation system}) is equivalent to the following equation system
\begin{equation}\label{psubproblem}
\left[\begin{array}{cc}
M_h & -\alpha K_h\\
{K_h}^T & M_h\\
\end{array}\right]
\left[\begin{array}{c}
{\tilde{p}}^k\\
{\tilde{y}}^k\\
\end{array}\right]=
\left[\begin{array}{c}
\mu^{k}-\alpha M_hf\\
M_hy_d-\lambda^{k}\\
\end{array}\right].
\end{equation}
Taking $A=M_h$, $B_1=B_2=K_h$, $a=1$ and $b=\alpha$, then it is clear that the coefficient matrix of (\ref{psubproblem}) has the form of {\cite[(1)]{Axelsson2016Comparison}} 
\begin{equation}
\mathscr{A}=\left[\begin{array}{cc}
A & -bB_2\\
aB_1 & A\\
\end{array}\right].
\end{equation}
In this paper, we employ the preconditioner 
\begin{equation}\label{preconditionerB}
\mathscr{B}=\left[\begin{array}{ccc}
A && -bB_2\\
aB_1 && A+\sqrt{ab}(B_1+B_2)\\
\end{array}\right],
\end{equation}
which was introduced in \cite{Axelsson2016Comparison} to precondition the generalized minimal residual (GMRES) method to solve (\ref{psubproblem}). We first present below some properties of the preconditioning matrix $\mathscr{B}$ and preconditioned matrix $\mathscr{B}^{-1}\mathscr{A}$, for more details, please refer to \cite{Axelsson2016Comparison}.
\begin{proposition}{\rm\cite[Proposition 1]{Axelsson2016Comparison}}
Consider a matrix $\mathscr{B}$ of form {\rm{(}\ref{preconditionerB}\rm{)}}. Let $H_i=A+\sqrt{ab}B_i$, $i=1,2$ be nonsingular. Then
\begin{equation*}
{\mathscr{B}}^{-1}=\left[\begin{array}{cc}
H_1^{-1}+H_2^{-1}-H_2^{-1}AH_1^{-1} & \sqrt{\frac{b}{a}}\left(I-H_2^{-1}A\right)H_1^{-1}\\
-\sqrt{\frac{b}{a}}H_2^{-1}\left(I-AH_1^{-1}\right) & H_2^{-1}AH_1^{-1}\\
\end{array}\right].
\end{equation*}
\end{proposition}
\begin{proposition}{\rm\cite[Proposition 2]{Axelsson2016Comparison}}\label{residualcompute}
Assume that $A+\sqrt{ab}B_i$, $i=1,2$ are nonsingular. Then $\mathscr{B}$ is nonsingular and a linear system with the preconditioner $\mathscr{B}$,
\begin{equation*}
\left[\begin{array}{cc}
A & -bB_2\\
aB_1 & A+\sqrt{ab}(B_1+B_2)\\
\end{array}\right]
\left[\begin{array}{c}
x\\
y\\
\end{array}\right]=
\left[\begin{array}{c}
f_1\\
f_2\\
\end{array}\right]
\end{equation*}
can be solved with only one solution with $A+\sqrt{ab}B_1$ and one with $A+\sqrt{ab}B_2$.
\end{proposition}
\begin{proposition}{\rm\cite[Proposition 4]{Axelsson2016Comparison}}
Let $\mathscr{A}=\left[\begin{array}{cc}
A & -bB^T\\
aB & A\\
\end{array}\right]$, where $a$, $b$ are nonzero and have the same sign and let $\mathscr{B}=\left[\begin{array}{cc}
A & -bB^T\\
aB & A+\sqrt{ab}(B+B^T)\\
\end{array}\right]$. If $ker(A)\cap ker(B)=\{0\}$ holds then the eigenvalues of $\mathscr{B}^{-1}\mathscr{A}$, are contained in the interval $[\frac{1}{2},1]$.
\end{proposition}

Numerical implementation of the preconditioning matrix $\mathscr{B}$ in Krylov subspace methods is realized by solving a sequence of generalized residual equations of the form
\begin{equation*}
\mathscr{B}v=r,
\end{equation*}
where $r=(r_1;r_2)\in{\mathbb{R}}^{2N_h}$, with $r_1$, $r_2\in{\mathbb{R}}^{N_h}$, represents the current residual vector and $v=(v_1;v_2)\in{\mathbb{R}}^{2N_h}$, with $v_1$, $v_2\in{\mathbb{R}}^{N_h}$, represents the generalized residual vector. Based on the proof of Proposition \ref{residualcompute} (see \cite{Axelsson2016Comparison} for more details), the computation of vector $v$ can take place using the following algorithm
\begin{algorithm}[H]
\caption{Solving the factorized operator}
\label{factorized operator}
1: Compute $g$ by solving the following linear system
\begin{equation*}
(M_h+\sqrt{\alpha}K_h)g=r_1+\sqrt{\alpha}r_2.
\end{equation*}

2: Compute $h$ by solving the following linear system
\begin{equation*}
(M_h+\sqrt{\alpha}K_h)h=r_1-M_hg.
\end{equation*}

3: Compute $v_1=g+h$ and $v_2=-\sqrt{\frac{1}{\alpha}}h$.
\end{algorithm}

We would like to point out that $Q:=M_h+\sqrt{\alpha}K_h$ is a symmetric positive definite matrix. If the Cholesky factorizations of $Q$, which only have to be done once, can be computed at a modest cost, then the two linear systems above can be solved exactly and effectively. However, if the Cholesky factorizations of $Q$ are not available, then we can use some alternative efficient methods, e.g., preconditioned conjugate gradient (PCG) method, Chebyshev semi-iteration or some multigrid schemes to solve them.

\subsection{Computation of the projection to set $\mathcal{C}$}\label{projection}
In this part, we focus on using Newton's method computing the projection to set $\mathcal{C}$. For a given $g$, computing $\mathrm{\Pi}_{\mathcal{C}}(g)$ is equivalent to solving the following optimization problem with inequality constraint
\begin{equation}\label{projection problem}
\begin{aligned}
\min\limits_{x}\quad &\frac{1}{2}\|x-g\|^2\\
{\rm{s.t.}}\quad&x^TD_hx\leq\delta.        
\end{aligned}
\end{equation}
For a given $g$, we first compute the value of $g^TD_hg$. If $g^TD_hg\leq\delta$, then the solution is $\bar{x}=g$. Otherwise, (\ref{projection problem}) can be transformed into a optimization problem with equality constraint 
\begin{equation}\label{projection problem 2}
\begin{aligned}
\min\limits_{x}\quad &\frac{1}{2}\|x-g\|^2\\
{\rm{s.t.}}\quad&x^TD_hx=\delta,      
\end{aligned}
\end{equation}
whose Lagrange function is
\begin{equation}
\bar{L}(x,\rho)=\frac{1}{2}\|x-g\|^2+(\rho,x^TD_hx-\delta),
\end{equation}
where $\rho\in\mathbb{R}$ is the Lagrange multiplier corresponding to the equality constraint. Then the optimality conditions of (\ref{projection problem 2}) are 
\begin{equation}
\left\{
\begin{aligned}
&\nabla_x \bar{L}\ =\ x-g+2\rho D_hx=0,\\
&\nabla_{\rho} \bar{L}\ =\ x^TD_hx-\delta=0,
\end{aligned}\right.
\end{equation}
which are actually nonlinear equations
\begin{equation}
H(x,\rho)=\left[\begin{array}{c}
x-g+2\rho D_hx\\
x^TD_hx-\delta\\
\end{array}\right]=0,
\end{equation}
whose Jacobian is
\begin{equation}
J(x,\rho)=\left[\begin{array}{ccc}
I+2\rho D_h &\ & 2D_hx\\
2x^TD_h && 0\\
\end{array}\right], 
\end{equation}
and residual is
\begin{equation}
r(x,\rho)=\left[\begin{array}{c}
x-g+2\rho D_hx\\
x^TD_hx-\delta\\
\end{array}\right].
\end{equation}
Then we give the framework of Newton's method for (\ref{projection problem 2}).
\begin{algorithm}[H]
\caption{Newton's method for (\ref{projection problem})}
\label{Newton for projection}
Choose $(x_0;\lambda_0)$;\\
\textbf{for} $k=0,1,2,\cdots$

\qquad Calculate a solution $p_k$ to the Newton equations
\begin{equation*}
\left[\begin{array}{ccc}
I+2\rho_k D_h &\ & 2D_hx_k\\
2x_k^TD_h && 0\\
\end{array}\right]p_k
=-\left[\begin{array}{c}
x_k-g+2\rho_k D_hx_k\\
x_k^TD_hx_k-\delta\\
\end{array}\right];
\end{equation*}
\qquad$(x_{k+1};\lambda_{k+1})=(x_{k};\lambda_{k})+p_k$;\\
\textbf{end(for)}
\end{algorithm}

\section{Numerical Experiment}\label{Numerical Experiment}
In this section, all calculations were performed using MATLAB (R2014a) on a PC with Intel (R) Xeon (R) CPU E5-2609 (2.50 GHz), whose operation system is 64-bit Windows 8.0 and RAM is 64.0 GB. The mass matrix, the stiffness matrix and the lump mass matrix are established by the iFEM software package \cite{Chen}.

For the FE-dABCD algorithm, the accuracy of a numerical solution is measured by the following residual
\begin{equation}
\eta_d=\max\{r_1,\ r_2,\ r_3,\ r_4\},
\end{equation}
where
\begin{equation*}
\begin{aligned}
&r_1=\|M_h(y-y_d)+K_h^Tp+\lambda\|/(1+\|M_hy_d\|),\\
&r_2=\|y-{\rm\Pi}_{\mathcal{C}}(y+\lambda)\|/(1+\|y\|),\\
&r_3=\|u-{\rm\Pi}_{\mathcal{S}}(u+\mu)\|/(1+\|u\|),\\
&r_4=\|K_hy-M_hu-M_hf\|/(1+\|M_hf\|).
\end{aligned}
\end{equation*}

To show the efficiency of the FE-dABCD algorithm, we will compare it with 
 alternating direction method of multipliers (ADMM)  and inexact heterogeneous alternating direction method of multipliers (ihADMM) (see \cite{Song2017A, Chen2016A}) applied to the primal problem. First we give the frame of classical ADMM algorithm and ihADMM algorithm for (\ref{eqn:discretized problem matrix_vector}).
 
\begin{algorithm}[H]\footnotesize
\caption{ADMM algorithm for (\ref{eqn:discretized problem matrix_vector})}
\label{algo:ADMM for problem P_h}
\leftline{Initialization: Give initial point $(\lambda^0, \mu^0, z^0, w^0)$ and a tolerant parameter $\tau >0$. Set $k=0$.}
\begin{description}
\item[\textbf{Step 1}] Compute $(y^{k+1},u^{k+1},p^{k+1})$ through solving the following equation system
\begin{equation*}
\left[
  \begin{array}{ccc}
    {M_h}+\sigma I & 0 &K_h^T \\
   0 & \alpha M_h+\sigma I & -{M_h} \\
   K_h & -M_h & 0 \\
  \end{array}
\right]\left[
         \begin{array}{c}
           y^{k+1} \\
           u^{k+1} \\
           p^{k+1} \\         
         \end{array}
       \right]
       =\left[
                 \begin{array}{c}
                   {M_h}y_d-{\lambda}^{k}+\sigma z^k \\
                   \sigma w^k-{\mu}^k \\
                   M_h f\\
                 \end{array}
               \right].
\end{equation*}
\item[\textbf{Step 2}] Compute $(z^{k+1},w^{k+1})$ as follows
       \begin{eqnarray*}
       z^{k+1}&={\rm\Pi}_{\mathcal{C}}\left(y^{k+1}+\frac{1}{\sigma}{\lambda}^k\right),\\
       w^{k+1}&={\rm\Pi}_{\mathcal{S}}\left(u^{k+1}+\frac{1}{\sigma}{\mu}^k\right).
       \end{eqnarray*}
\item[\textbf{Step 3}] Compute $({\lambda}^{k+1},\mu^{k+1})$ as follows
\begin{eqnarray*}
    \lambda^{k+1} &= \lambda^k+\sigma(y^{k+1}-z^{k+1}),\\
    \mu^{k+1} &= \mu^k+\sigma(u^{k+1}-w^{k+1}).
\end{eqnarray*}
\item[\textbf{Step 4}] If a termination criterion is met, Stop; else, set $k:=k+1$ and go to Step 1.
\end{description}
\end{algorithm}

\begin{algorithm}[H]\footnotesize
\caption{ihADMM algorithm for (\ref{eqn:discretized problem matrix_vector})}
\label{algo:ihADMM for problem P_h}
\leftline{Initialization: Give initial point $(\lambda^0, \mu^0, z^0, w^0)$ and a tolerant parameter $\tau >0$. Set $k=0$.}
\begin{description}
\item[\textbf{Step 1}] Compute $(y^{k+1},p^{k+1})$ through solving the following equation system
\begin{equation*}
\left[
  \begin{array}{cc}
    (1+\sigma) {M_h} & K_h^T \\
    K_h & -\frac{1}{\alpha+\sigma}{M_h} \\
  \end{array}
\right]\left[
         \begin{array}{c}
           y^{k+1} \\
           p^{k+1} \\
         \end{array}
       \right]
       =\left[
                 \begin{array}{c}
                   {M_h}(y_d-{\lambda}^k+\sigma z^k) \\
                   \frac{1}{\alpha+\sigma}M_h(\sigma w^k-\mu^k)+M_hf\\
                 \end{array}
               \right].
\end{equation*}
Compute $u^{k+1}$ as follows
\begin{equation*}
u^{k+1}=\frac{1}{\sigma+\alpha}(p^{k+1}+\sigma w^k-\mu^k).
\end{equation*}
\item[\textbf{Step 2}] Compute $(z^{k+1},w^{k+1})$ as follows
       \begin{eqnarray*}
       z^{k+1}&={\rm\Pi}_{\mathcal{C}}\left(y^{k+1}+\frac{1}{\sigma}W_h^{-1}M_h{\lambda}^k\right),\\
       w^{k+1}&={\rm\Pi}_{\mathcal{S}}\left(u^{k+1}+\frac{1}{\sigma}W_h^{-1}M_h{\mu}^k\right).
       \end{eqnarray*}
\item[\textbf{Step 3}] Compute $({\lambda}^{k+1},\mu^{k+1})$ as follows
\begin{eqnarray*}
    \lambda^{k+1} &= \lambda^k+\sigma(y^{k+1}-z^{k+1}),\\
    \mu^{k+1} &= \mu^k+\sigma(u^{k+1}-w^{k+1}).
\end{eqnarray*}
\item[\textbf{Step 4}] If a termination criterion is met, Stop; else, set $k:=k+1$ and go to Step 1.
\end{description}
\end{algorithm}

For the classical ADMM algorithm, the accuracy of a numerical solution is measured by the following residual
\begin{equation}
\eta_c=\max\{\gamma_1,\ \gamma_2,\ \gamma_3,\ \gamma_4,\ \gamma_5,\ \gamma_6,\ \gamma_7\},
\end{equation}
where
\begin{equation*}
\begin{aligned}
&\gamma_1=\|M_h(y-y_d)+K_h^Tp+\lambda\|/(1+\|M_hy_d\|),\\
&\gamma_2=\|\alpha M_h u-M_h^Tp+\mu\|/(1+\|M_hu\|),\\
&\gamma_3=\|K_hy-M_hu-M_hf\|/(1+\|M_hf\|),\\
&\gamma_4=\|y-z\|/(1+\|y\|),\\
&\gamma_5=\|u-w\|/(1+\|u\|),\\
&\gamma_6=\|z-{\rm\Pi}_{\mathcal{C}}(z+\lambda)\|/(1+\|z\|),\\
&\gamma_7=\|w-{\rm\Pi}_{\mathcal{S}}(w+\mu)\|/(1+\|w\|).
\end{aligned}
\end{equation*}
For the inexact heterogeneous ADMM (ihADMM) algorithm, the accuracy of a numerical solution is measured by the following residual
\begin{equation}
\eta_h=\max\{\zeta_1,\ \zeta_2,\ \zeta_3,\ \zeta_4,\ \zeta_5,\ \zeta_6,\ \zeta_7\},
\end{equation}
where
\begin{equation*}
\begin{aligned}
&\zeta_1=\|M_h(y-y_d)+K_h^Tp+M_h\lambda\|/(1+\|M_hy_d\|),\\
&\zeta_2=\|\alpha M_h u-M_h^Tp+M_h\mu\|/(1+\|M_hu\|),\\
&\zeta_3=\|K_hy-M_hu-M_hf\|/(1+\|M_hf\|),\\
&\zeta_4=\|M_h(y-z)\|/(1+\|y\|),\\
&\zeta_5=\|M_h(u-w)\|/(1+\|u\|),\\
&\zeta_6=\|z-{\rm\Pi}_{\mathcal{C}}(z+M_h\lambda)\|/(1+\|z\|),\\
&\zeta_7=\|w-{\rm\Pi}_{\mathcal{S}}(w+M_h\mu)\|/(1+\|w\|).
\end{aligned}
\end{equation*}
Let $\epsilon$ be a given accuracy tolerance and $k_m$ be a given maximum iteration times, then the terminal condition is $\eta_d\ (\eta_c,\ \eta_h)<\epsilon$ or $k>k_m$. 

There are two examples in this section. In the first example, the exact control and exact state are known, while for the second one, only the desired state $y_d$ is known. In both two examples, we compare FE-dABCD algorithm with ihADMM algorithm and ADMM algorithm on some convergence behavior, including the times of iteration, residual $\eta$ and CPU time. In both two examples, `$\#$dofs' denotes the dimension of the control variable on each grid level, `iter' represents the times of iteration and `residual' represents the precision $\eta_d\ (\eta_c,\ \eta_h)$ of the numerical algorithm, which are defined above. 

\begin{example}\label{example1}
We now consider problem (\ref{original problem}) with the following data, $\mathrm{\Omega}=B_2(0)\subset \mathbb{R}^2$, $\alpha=1$,
\begin{equation*}
\mathcal{W}=\{u\in L^{\infty}(\mathrm{\Omega})|\ -2\leq u \leq 2\quad\mathrm{a.e.}\ x \in \mathrm{\Omega}\},\quad\mathcal{K}=\{\mathbf{z}\in {L^{2}(\bar{\mathrm{\Omega}})}^{2}|\int_{\mathrm{\Omega}}\ |\mathbf{z}(x)|^2dx\leq2\}
\end{equation*}
as well as
\begin{equation*}
 y_d(x):=\left\{ \begin{aligned}
         \frac{1}{4}+\frac{1}{2}\log 2-\frac{1}{4}|x|^2,&\qquad 0\leq|x|\leq1, \\
         \frac{1}{2}\log 2-\frac{1}{2}\log |x|,\ \ &\qquad 1\leq|x|\leq2.
                          \end{aligned} \right.
\end{equation*}
We consider the problem
\begin{equation*}
\begin{aligned}
                -\mathrm{\Delta} y&=u+f\quad\mathrm{in}\  \mathrm{\Omega}, \\
         y&=0\ \ \qquad\mathrm{on}\ \mathrm{\Gamma},\\
\end{aligned}
\end{equation*}
where
\begin{equation*}
 f(x):=\left\{ \begin{aligned}
         \frac{5}{4}+\frac{1}{2}\log 2-\frac{1}{4}|x|^2,&\qquad 0\leq|x|\leq1, \\
         \frac{1}{2}\log 2-\frac{1}{2}\log |x|,\ \ &\qquad 1\leq|x|\leq2.
                          \end{aligned} \right.
\end{equation*}
The optimization problem then has the unique solution
\begin{equation*}
 u(x):=\left\{ \begin{aligned}
         -\frac{1}{4}-\frac{1}{2}\log 2+\frac{1}{4}|x|^2,&\qquad 0\leq|x|\leq1, \\
         -\frac{1}{2}\log 2+\frac{1}{2}\log |x|,\ \ &\qquad 1\leq|x|\leq2.
                          \end{aligned} \right.
\end{equation*}
with the corresponding state $y\equiv y_d$. It is easy to see that the bounds on the control are not active, then from (\ref{optimal condition 5}) we obtain that $p=-u$.

\begin{table}[H]\small
\ \\
\ \\
\caption{The convergence behavior of FE-dABCD algorithm, ihADMM algorithm and ADMM algorithm for Example \ref{example1}.}
\label{table1}
\begin{center}
\begin{tabular}{@{\extracolsep{\fill}}ccccccccccccccccc}
\hline
\multirow{2}{*}{$h$}                   &&&\multirow{2}{*}{$\#$dofs}     &&&                                                    &&& \multirow{2}{*}{FE-dABCD}&&& \multirow{2}{*}{ihADMM}&&& \multirow{2}{*}{ADMM}\\
&&&&&&&&&&&&&&&\\
\hline
                                                  &&&                                            &&&\multirow{2}{*}{iter}                      &&&\multirow{2}{*}{15}&&&\multirow{2}{*}{33}&&&\multirow{2}{*}{33}\\
\multirow{2}{*}{$\frac{1}{2^3}$}  &&&\multirow{2}{*}{273}            &&&\multirow{2}{*}{residual $\eta$}    &&&\multirow{2}{*}{8.31e-05}&&&\multirow{2}{*}{9.86e-05}&&&\multirow{2}{*}{9.17e-05}\\
                                                  &&&                                           &&&\multirow{2}{*}{time/s}                   &&&\multirow{2}{*}{0.1973}&&&\multirow{2}{*}{0.3298}&&&\multirow{2}{*}{0.3544}\\
&&&&&&&&&&&&&&&\\

                                                  &&&                                            &&&\multirow{2}{*}{iter}                      &&&\multirow{2}{*}{13}&&&\multirow{2}{*}{32}&&&\multirow{2}{*}{45}\\
\multirow{2}{*}{$\frac{1}{2^4}$}  &&&\multirow{2}{*}{1145}          &&&\multirow{2}{*}{residual $\eta$}    &&&\multirow{2}{*}{9.58e-05}&&&\multirow{2}{*}{9.57e-05}&&&\multirow{2}{*}{9.82e-05}\\
                                                  &&&                                           &&&\multirow{2}{*}{time/s}                   &&&\multirow{2}{*}{0.8488}&&&\multirow{2}{*}{1.9684}&&&\multirow{2}{*}{2.9272}\\
&&&&&&&&&&&&&&&\\
                                                  &&&                                            &&&\multirow{2}{*}{iter}                      &&&\multirow{2}{*}{14}&&&\multirow{2}{*}{32}&&&\multirow{2}{*}{56}\\
\multirow{2}{*}{$\frac{1}{2^5}$}  &&&\multirow{2}{*}{4689}          &&&\multirow{2}{*}{residual $\eta$}    &&&\multirow{2}{*}{9.54e-05}&&&\multirow{2}{*}{9.12e-05}&&&\multirow{2}{*}{9.92e-05}\\
                                                  &&&                                           &&&\multirow{2}{*}{time/s}                   &&&\multirow{2}{*}{10.1207}&&&\multirow{2}{*}{19.5189}&&&\multirow{2}{*}{34.7385}\\
&&&&&&&&&&&&&&&\\
                                                  &&&                                            &&&\multirow{2}{*}{iter}                      &&&\multirow{2}{*}{15}&&&\multirow{2}{*}{33}&&&\multirow{2}{*}{71}\\
\multirow{2}{*}{$\frac{1}{2^6}$}  &&&\multirow{2}{*}{18977}        &&&\multirow{2}{*}{residual $\eta$}    &&&\multirow{2}{*}{9.58e-05}&&&\multirow{2}{*}{8.09e-05}&&&\multirow{2}{*}{9.98e-05}\\
                                                  &&&                                           &&&\multirow{2}{*}{time/s}                   &&&\multirow{2}{*}{72.2546}&&&\multirow{2}{*}{107.985}&&&\multirow{2}{*}{237.331}\\
&&&&&&&&&&&&&&&\\
                                                  &&&                                            &&&\multirow{2}{*}{iter}                      &&&\multirow{2}{*}{14}&&&\multirow{2}{*}{32}&&&\multirow{2}{*}{88}\\
\multirow{2}{*}{$\frac{1}{2^7}$}  &&&\multirow{2}{*}{76353}        &&&\multirow{2}{*}{residual $\eta$}    &&&\multirow{2}{*}{8.04e-05}&&&\multirow{2}{*}{5.71e-05}&&&\multirow{2}{*}{5.88e-05}\\
                                                  &&&                                           &&&\multirow{2}{*}{time/s}                 &&&\multirow{2}{*}{854.684}&&&\multirow{2}{*}{1169.12}&&&\multirow{2}{*}{3259.29}\\
&&&&&&&&&&&&&&&\\
\hline
\end{tabular}
\end{center}
\end{table}

In this numerical experiment, we first focus on the convergence behavior of FE-dABCD algorithm compared with ihADMM algorithm and ADMM algorithm. In this case, we set $\epsilon=10^{-4}$ and $k_m=100$. That is to say, the algorithm is terminated when $\eta_d\ (\eta_c,\ \eta_h)<10^{-4}$ or ${\rm{iter}}>100$. The corresponding convergence behavior, including the times of iteration, residual $\eta_d\ (\eta_c,\ \eta_h)$ and CPU time, of FE-dABCD algorithm, ihADMM algorithm and ADMM algorithm are given in Table \ref{table1}. And as an example, the figures of exact state $y$, numerical state $y_{h}$ and exact control $u$, numerical control $u_{h}$ on the grid of size $h=\frac{1}{2^5}$ are displayed in Figure \ref{example:y} and Figure \ref{example:u} respectively. Then we test this problem with different values of $\alpha$ on the grid of size $h=\frac{1}{2^6}$ to show the robustness of our proposed FE-dABCD algorithm. In this case, we still use the $y_d$ defined above and set $a=-\frac{1}{2}$, $b=\frac{1}{2}$, $\delta=1$, $\epsilon=10^{-4}$, $k_m=100$ and let $\alpha$ range from $10^{-2}$ to $10^{-5}$.

The results in Table \ref{table1} show that the number of iteration of FE-dABCD algorithm is independent of the discretization level. It is easy to see from Table \ref{table1} that the number of iteration of FE-dABCD for five discretization levels are 15, 13, 14, 15 and 14 respectively. From Table \ref{table1}, we can also verify the efficiency of our proposed FE-dABCD algorithm. The results for testing the problem with different values of $\alpha$ are presented in Table \ref{table2}. Although as $\alpha$ changes from $10^{-2}$ to $10^{-5}$, the number of iteration of the FE-dABCD algorithm increases, it does not change that dramatically. The FE-dABCD algorithm could solve (\ref{dual problem}) for all tested values of $\alpha$ in $30$ iterations, which shows the robustness of FE-dABCD algorithm with respect to $\alpha$.

\begin{figure}[H]
\centering
\subfigure[exact state $y$ ]{
\includegraphics[width=0.49\textwidth]{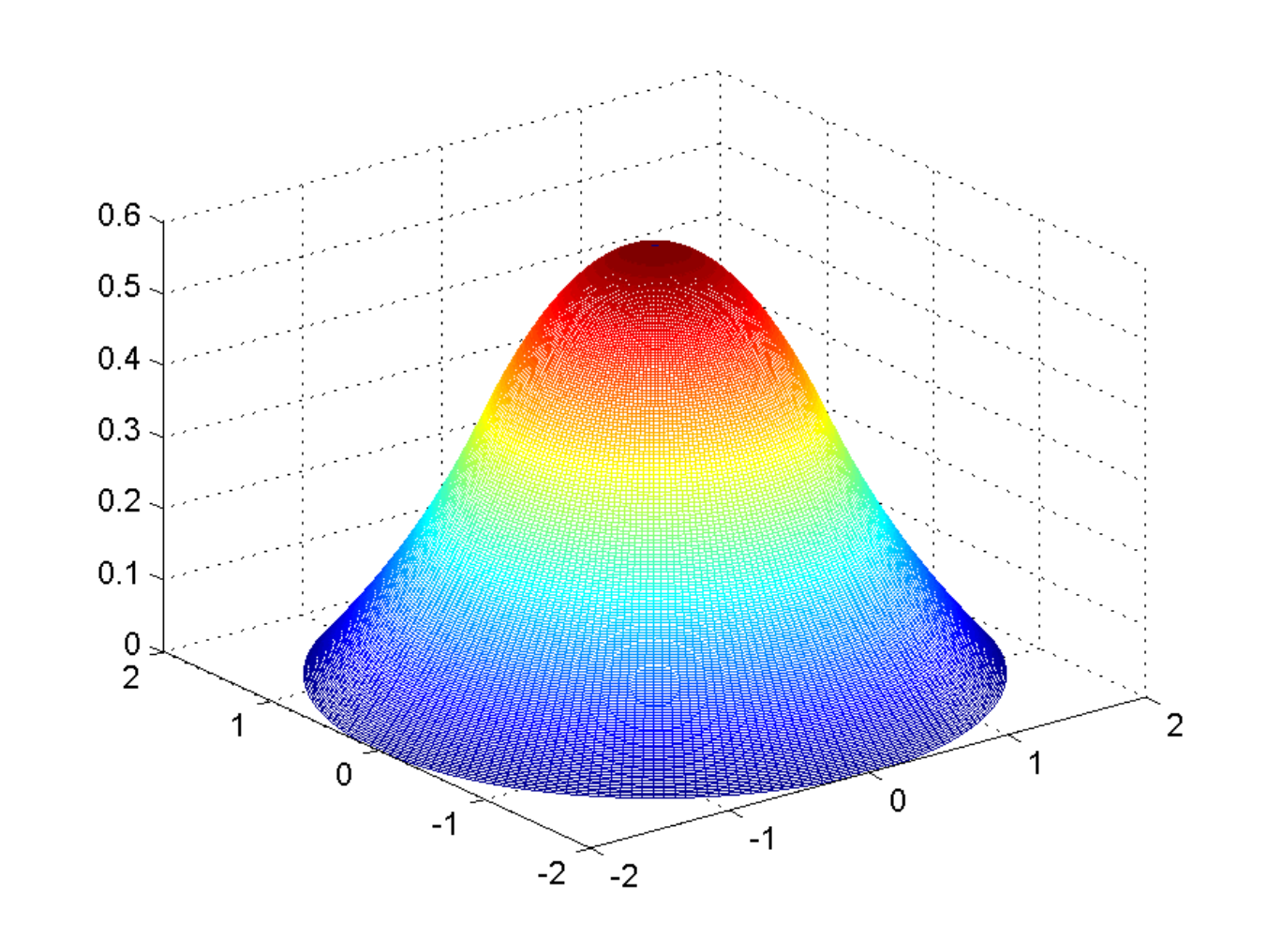}}
\subfigure[numerical state $y_{h}$ ]{
\includegraphics[width=0.49\textwidth]{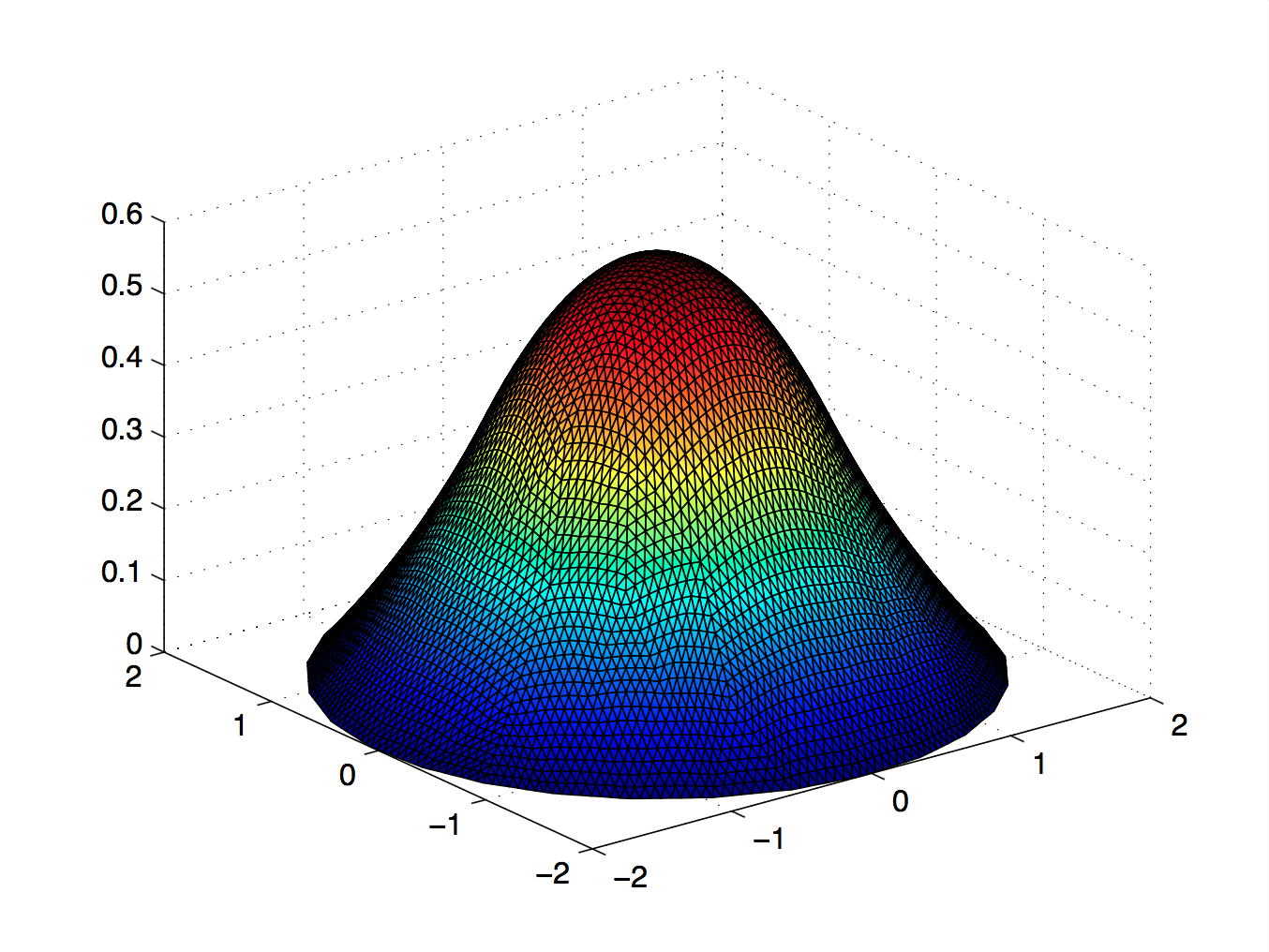}}
\caption{Figures of exact and numerical state on the grid of size $h=\frac{1}{2^5}$ for Example \ref{example1}.}
\label{example:y}
\end{figure}

\begin{figure}[H]
\centering
\subfigure[exact control $u$ ]{
\includegraphics[width=0.49\textwidth]{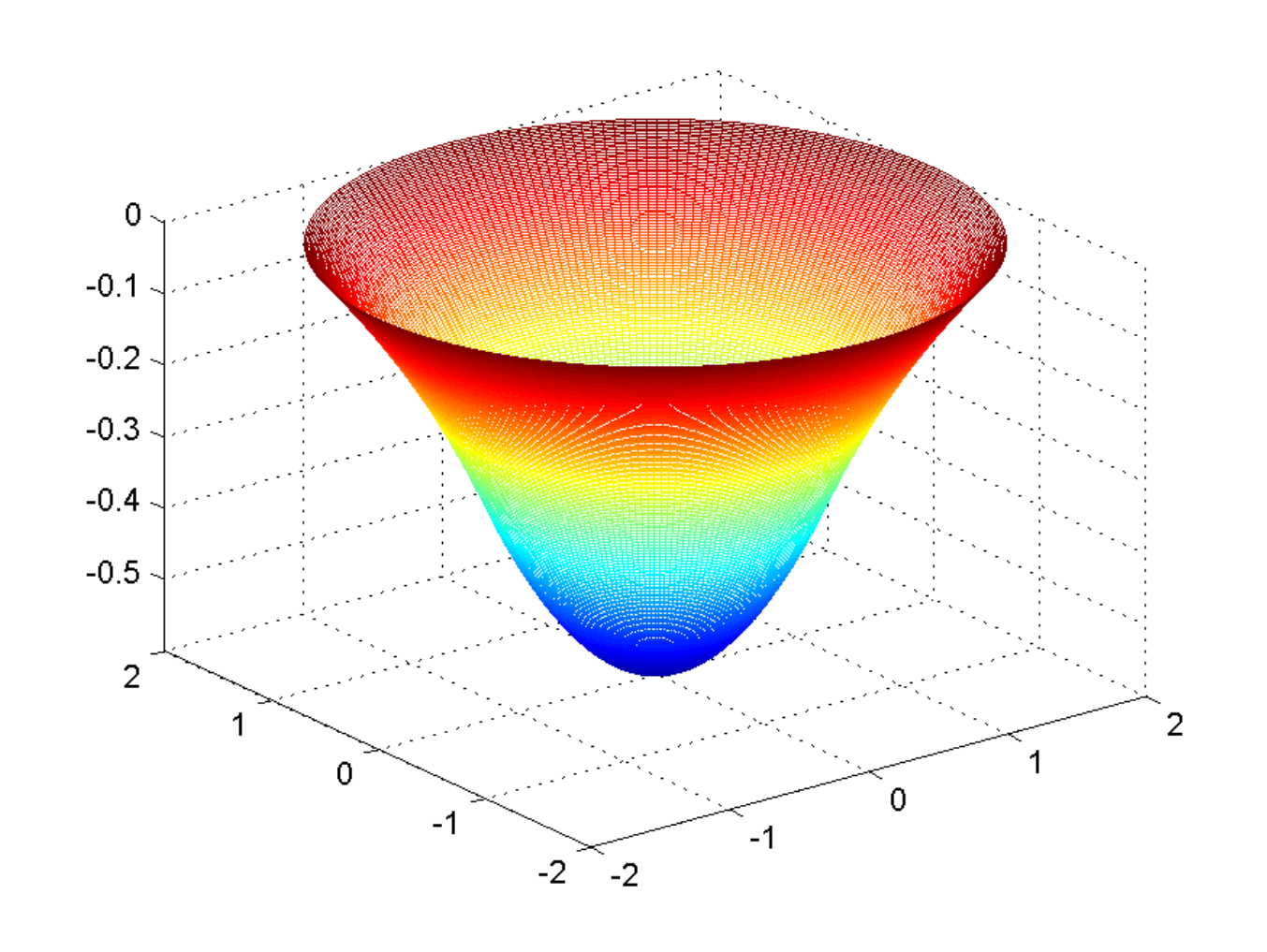}}
\subfigure[numerical control $u_{h}$ ]{
\includegraphics[width=0.49\textwidth]{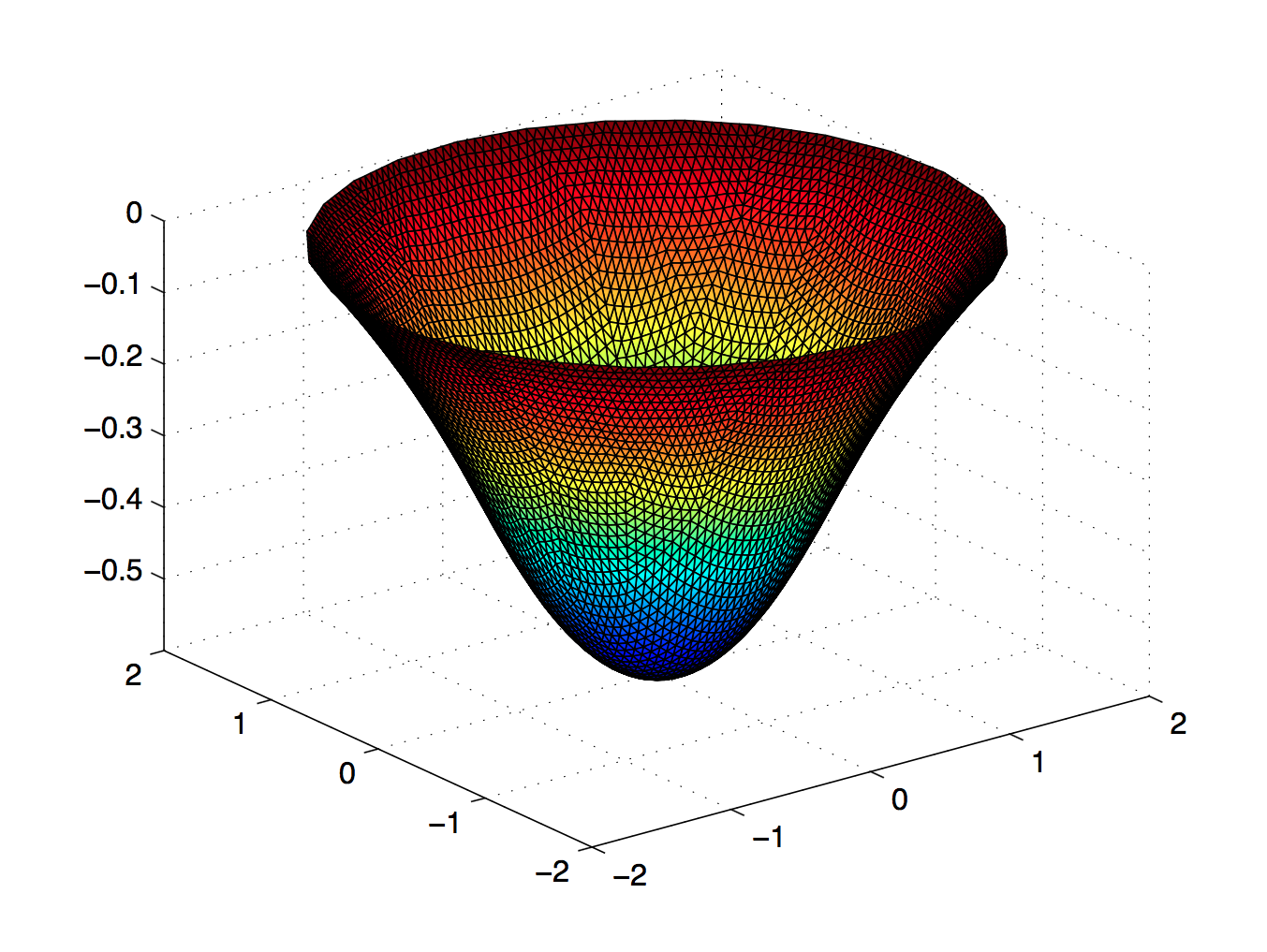}}
\caption{Figures of exact and numerical control on the grid of size $h=\frac{1}{2^5}$ for Example \ref{example1}.}
\label{example:u}
\end{figure}

\begin{table}[H]\small
\ \\
\ \\
\caption{The performance of FE-dABCD algorithm for (\ref{dual problem}) with different values of $\alpha$.}
\label{table3}
\begin{center}
\begin{tabular}{@{\extracolsep{\fill}}cccccccccc}
\hline
\multirow{2}{*}{$h$}                   &&&\multirow{2}{*}{$\alpha$}    &&& \multirow{2}{*}{iter}&&& \multirow{2}{*}{residual $\eta_d$}\\
&&&&&&&&&\\
\hline
&&&&&&&&&\\
\multirow{2}{*}{$\frac{1}{2^6}$}  &&&\multirow{2}{*}{$10^{-2}$}                &&&\multirow{2}{*}{16}&&&\multirow{2}{*}{9.28e-05}\\
&&&&&&&&&\\
&&&&&&&&&\\
                                                  &&&\multirow{2}{*}{$5\times10^{-3}$}           &&&\multirow{2}{*}{17}&&&\multirow{2}{*}{9.55e-05}\\
&&&&&&&&&\\
&&&&&&&&&\\
                                                  &&&\multirow{2}{*}{$10^{-3}$}          &&&\multirow{2}{*}{19}&&&\multirow{2}{*}{8.48e-05}\\
&&&&&&&&&\\
&&&&&&&&&\\
                                                  &&&\multirow{2}{*}{$5\times10^{-4}$}            &&&\multirow{2}{*}{20}&&&\multirow{2}{*}{7.47e-05}\\
&&&&&&&&&\\
&&&&&&&&&\\
                                                  &&&\multirow{2}{*}{$10^{-4}$}            &&&\multirow{2}{*}{21}&&&\multirow{2}{*}{9.27e-05}\\
&&&&&&&&&\\
&&&&&&&&&\\
                                                  &&&\multirow{2}{*}{$5\times10^{-5}$}            &&&\multirow{2}{*}{22}&&&\multirow{2}{*}{9.98e-05}\\
&&&&&&&&&\\
&&&&&&&&&\\
                                                  &&&\multirow{2}{*}{$10^{-5}$}            &&&\multirow{2}{*}{25}&&&\multirow{2}{*}{3.81e-05}\\
&&&&&&&&&\\
&&&&&&&&&\\
\hline
\end{tabular}
\end{center}
\end{table}

\end{example}

\begin{example}\label{example2}
In this example, we consider problem (\ref{original problem}) with the following data, $\mathrm{\Omega}=B_1(0)\subset \mathbb{R}^2$,
\begin{equation*}
\mathcal{W}=\{u\in L^{\infty}(\mathrm{\Omega})|\ 0\leq u \leq \frac{1}{2}\quad\mathrm{a.e.}\ x \in \mathrm{\Omega}\},\quad\mathcal{K}=\{\mathbf{z}\in {L^{2}(\bar{\mathrm{\Omega}})}^{2}|\int_{\mathrm{\Omega}}\ |\mathbf{z}(x)|^2dx\leq\frac{1}{2}\}
\end{equation*}
as well as
\begin{equation*}
y_d(x):=|x|^2,\qquad 0\leq|x|\leq1.
\end{equation*}
We consider the problem
\begin{equation*}
\begin{aligned}
                -\mathrm{\Delta} y&=u\quad\mathrm{in}\  \mathrm{\Omega}, \\
         y&=0\quad\mathrm{on}\ \mathrm{\Gamma},\\
\end{aligned}
\end{equation*}
which means $f(x)=0$.

In this example, we still first focus on the convergence behavior of FE-dABCD algorithm compared with ihADMM algorithm and ADMM algorithm. In this case, we set $\alpha=10^{-2}$, $\epsilon=10^{-4}$ and $k_m=100$. The relative results are given in Table \ref{table2}. Then, similar to Example \ref{example1}, to show the robustness of our proposed FE-dABCD algorithm with respect to the parameter $\alpha$, we will also test the same problem with different values of $\alpha$, ranging from $10^{-2}$ to $10^{-5}$, on the grid of size $h=\frac{1}{2^6}$ and the corresponding results are presented in Table \ref{table4}. 

It is clear from Table \ref{table2} that the efficiency of our proposed FE-dABCD algorithm compared with ihADMM algorithm and ADMM algorithm. We can also see from the results in Table \ref{table2} that the number of iteration of FE-dABCD algorithm is independent of the discretization level. And from the results in Table \ref{table4}, we can find that although the number of iterations of our proposed FE-dABCD algorithm increases obviously when $\alpha$ changes from $10^{-2}$ to $10^{-5}$, it still could solve problem (\ref{dual problem}) for all tested values of $\alpha$ in $50$ iterations.

\begin{table}[H]\small
\ \\
\ \\
\caption{The convergence behavior of FE-dABCD algorithm, ihADMM algorithm and ADMM algorithm for Example \ref{example2}.}
\label{table2}
\begin{center}
\begin{tabular}{@{\extracolsep{\fill}}ccccccccccccccccc}
\hline
\multirow{2}{*}{$h$}                   &&&\multirow{2}{*}{$\#$dofs}     &&&                                                    &&& \multirow{2}{*}{FE-dABCD}&&& \multirow{2}{*}{ihADMM}&&& \multirow{2}{*}{ADMM}\\
&&&&&&&&&&&&&&&\\
\hline
                                                  &&&                                            &&&\multirow{2}{*}{iter}                      &&&\multirow{2}{*}{10}&&&\multirow{2}{*}{30}&&&\multirow{2}{*}{30}\\
\multirow{2}{*}{$\frac{1}{2^3}$}  &&&\multirow{2}{*}{273}            &&&\multirow{2}{*}{residual $\eta$}    &&&\multirow{2}{*}{7.83e-05}&&&\multirow{2}{*}{2.74e-05}&&&\multirow{2}{*}{7.74e-05}\\
                                                  &&&                                           &&&\multirow{2}{*}{time/s}                   &&&\multirow{2}{*}{0.0589}&&&\multirow{2}{*}{0.0669}&&&\multirow{2}{*}{0.0692}\\
&&&&&&&&&&&&&&&\\

                                                  &&&                                            &&&\multirow{2}{*}{iter}                      &&&\multirow{2}{*}{9}&&&\multirow{2}{*}{33}&&&\multirow{2}{*}{36}\\
\multirow{2}{*}{$\frac{1}{2^4}$}  &&&\multirow{2}{*}{1145}          &&&\multirow{2}{*}{residual $\eta$}    &&&\multirow{2}{*}{3.56e-05}&&&\multirow{2}{*}{9.21e-05}&&&\multirow{2}{*}{2.06e-05}\\
                                                  &&&                                           &&&\multirow{2}{*}{time/s}                   &&&\multirow{2}{*}{0.2437}&&&\multirow{2}{*}{0.3238}&&&\multirow{2}{*}{0.4969}\\
&&&&&&&&&&&&&&&\\
                                                  &&&                                            &&&\multirow{2}{*}{iter}                      &&&\multirow{2}{*}{8}&&&\multirow{2}{*}{32}&&&\multirow{2}{*}{47}\\
\multirow{2}{*}{$\frac{1}{2^5}$}  &&&\multirow{2}{*}{4689}          &&&\multirow{2}{*}{residual $\eta$}    &&&\multirow{2}{*}{7.57e-05}&&&\multirow{2}{*}{9.66e-05}&&&\multirow{2}{*}{8.96e-05}\\
                                                  &&&                                           &&&\multirow{2}{*}{time/s}                   &&&\multirow{2}{*}{0.9482}&&&\multirow{2}{*}{1.7690}&&&\multirow{2}{*}{3.6804}\\
&&&&&&&&&&&&&&&\\
                                                  &&&                                            &&&\multirow{2}{*}{iter}                      &&&\multirow{2}{*}{9}&&&\multirow{2}{*}{31}&&&\multirow{2}{*}{55}\\
\multirow{2}{*}{$\frac{1}{2^6}$}  &&&\multirow{2}{*}{18977}        &&&\multirow{2}{*}{residual $\eta$}    &&&\multirow{2}{*}{3.39e-05}&&&\multirow{2}{*}{9.77e-05}&&&\multirow{2}{*}{2.17e-05}\\
                                                  &&&                                           &&&\multirow{2}{*}{time/s}                   &&&\multirow{2}{*}{6.2710}&&&\multirow{2}{*}{8.2806}&&&\multirow{2}{*}{24.7473}\\
&&&&&&&&&&&&&&&\\
                                                  &&&                                            &&&\multirow{2}{*}{iter}                      &&&\multirow{2}{*}{7}&&&\multirow{2}{*}{28}&&&\multirow{2}{*}{69}\\
\multirow{2}{*}{$\frac{1}{2^7}$}  &&&\multirow{2}{*}{76353}        &&&\multirow{2}{*}{residual $\eta$}    &&&\multirow{2}{*}{8.36e-05}&&&\multirow{2}{*}{9.35e-05}&&&\multirow{2}{*}{6.49e-05}\\
                                                  &&&                                           &&&\multirow{2}{*}{time/s}                 &&&\multirow{2}{*}{32.7311}&&&\multirow{2}{*}{40.2470}&&&\multirow{2}{*}{217.994}\\
&&&&&&&&&&&&&&&\\
\hline
\end{tabular}
\end{center}
\end{table}

\begin{table}[H]\small
\ \\
\ \\
\caption{The performance of FE-dABCD algorithm for (\ref{dual problem}) with different values of $\alpha$.}
\label{table4}
\begin{center}
\begin{tabular}{@{\extracolsep{\fill}}cccccccccc}
\hline
\multirow{2}{*}{$h$}                   &&&\multirow{2}{*}{$\alpha$}    &&& \multirow{2}{*}{iter}&&& \multirow{2}{*}{residual $\eta_d$}\\
&&&&&&&&&\\
\hline
&&&&&&&&&\\
\multirow{2}{*}{$\frac{1}{2^6}$}  &&&\multirow{2}{*}{$10^{-2}$}                &&&\multirow{2}{*}{9}&&&\multirow{2}{*}{3.39e-05}\\
&&&&&&&&&\\
&&&&&&&&&\\
                                                  &&&\multirow{2}{*}{$5\times10^{-3}$}           &&&\multirow{2}{*}{11}&&&\multirow{2}{*}{4.25e-05}\\
&&&&&&&&&\\
&&&&&&&&&\\
                                                  &&&\multirow{2}{*}{$10^{-3}$}          &&&\multirow{2}{*}{15}&&&\multirow{2}{*}{7.38e-05}\\
&&&&&&&&&\\
&&&&&&&&&\\
                                                  &&&\multirow{2}{*}{$5\times10^{-4}$}            &&&\multirow{2}{*}{21}&&&\multirow{2}{*}{7.71e-05}\\
&&&&&&&&&\\
&&&&&&&&&\\
                                                  &&&\multirow{2}{*}{$10^{-4}$}            &&&\multirow{2}{*}{30}&&&\multirow{2}{*}{7.38e-05}\\
&&&&&&&&&\\
&&&&&&&&&\\
                                                  &&&\multirow{2}{*}{$5\times10^{-5}$}            &&&\multirow{2}{*}{39}&&&\multirow{2}{*}{7.71e-05}\\
&&&&&&&&&\\
&&&&&&&&&\\
                                                  &&&\multirow{2}{*}{$10^{-5}$}            &&&\multirow{2}{*}{46}&&&\multirow{2}{*}{7.38e-05}\\
&&&&&&&&&\\
&&&&&&&&&\\
\hline
\end{tabular}
\end{center}
\end{table}

\end{example}

\section{Conclusion}\label{Conclusion}
In this paper, we impose integral constraint on the gradient of the state and box constraints on the control. Our main results are proving the optimal conditions for the optimal control problem and also giving an efficient finite element duality-based inexact majorized accelerated block coordinate descent (FE-dABCD) algorithm. We consider piecewise linear approximation for both the state and the control and then transform the discretized problem into a multi-block unconstrained optimization problem by its dual. Our proposed method employs a majorization technique, which allow us to flexibly choose different proximal terms for different subproblems. Additionally, each subproblem only has to be solved approximately thanks to the inexactness of our proposed method. These flexibilities make each subproblem keep good structure and can be solved efficiently and reduce the cost for solving the subproblems largely, which improve the efficiency of our proposed method greatly. Specifically, we solve the smooth subproblem by GMRES method with preconditioner and solve nonsmooth subproblems through introducing appropriate proximal terms and semi-smooth Newton (SSN) method. We proved that the FE-dABCD algorithm enjoys $O(\frac{1}{k^2})$ iteration complexity. It is also easy to see the efficiency of FE-dABCD algorithm from the numerical results.

\begin{acknowledgements}
We would like to thank Prof. Long Chen very much for the contribution of the FEM package iFEM \cite{Chen} in Matlab. 
\end{acknowledgements}



\end{document}